\documentclass[12pt, oneside,reqno]{amsart}
\usepackage{amsmath, amsfonts, amssymb}
\usepackage{eucal}
\usepackage{mathrsfs}
\usepackage{latexsym}
\usepackage{bbm}
\usepackage{mathtools}
\usepackage{exscale}
\usepackage{cases}
\usepackage{xcolor}
\usepackage{comment}

\usepackage[pdfstartview=FitH,
            CJKbookmarks=true,
            bookmarksnumbered=true,
            bookmarksopen=true,
            colorlinks=true,
            linkcolor=blue,
            anchorcolor=blue,
            citecolor=red,
            urlcolor=blue
            ]{hyperref}

\numberwithin{equation}{section}
\newtheorem{theorem}{Theorem}[section]
\newtheorem{definition}[theorem]{Definition}
\newtheorem{lemma}[theorem]{Lemma}
\newtheorem{corollary}[theorem]{Corollary}
\newtheorem{proposition}[theorem]{Proposition}

\newtheorem{definition and theorem}[theorem]{Definition and Theorem}
\def\bl{\begin{lemma}}
\def\el{\end{lemma}}
\def\bc{\begin{corollary}}
\def\ec{\end{corollary}}
\def\bt{\begin{theorem}}
\def\et{\end{theorem}}
\def\bp{\begin{proposition}}
\def\ep{\end{proposition}}
\def\be{\begin{equation}}
\def\ee{\end{equation}}
\def\baa{\begin{align*}}
\def\eaa{\end{align*}}

\def\bd{\begin{definition}}
\def\ed{\end{definition}}

\allowdisplaybreaks[4]
\theoremstyle{plain}
\newtheorem{thm}{Theorem}[section]
\newtheorem{ex}[thm]{Example}
\newtheorem*{Thm1}{Theorem 1}
\newtheorem*{Thm2}{Theorem 2}
\newtheorem*{Thm3}{Theorem 3}

\newtheorem{lem}[thm]{Lemma}

\newtheorem{defn}[thm]{Definition}

\theoremstyle{remark}

\newcommand{\lsub}[1]{\hskip -1.0pt\lower.3ex\hbox{$_{#1}$}}

\theoremstyle{definition}

\newcommand{\sn}{\mathbb S^{n-1}}

\newcommand{\tr}{\mathbb R}
\newcommand{\rn}{\mathbb R^n}

\newcommand{\tH}{\mathcal{H}}

\newcommand{\di}{\int_{\rn}\int_{\rn}}

\newcommand{\rR}{{\rm R}}
\newcommand{\Ra}{{\rm R}_{\alpha}}
\newcommand{\Raf}{{\rm R}_{\alpha} f}

\newcommand{\Pia}{\Pi^{*,-\alpha}}
\newcommand{\Laf}{{\rm L}_{\alpha} f}
\newcommand{\La}{{\rm L}_{\alpha}}
\newcommand{\supp}{{\rm supp}}

\renewcommand{\chi}{\operatorname{1}}

\title[Affine chord Sobolev inequalities and radial mean bodies for functions]{Affine  chord Sobolev inequalities and radial mean bodies for functions}
\author{Fernanda M. Ba$\hat{\rm E}$ta \quad  and  \quad Xiaxing Cai}

\begin{document}

\begin{abstract}
Affine isoperimetric inequalities for functional radial mean bodies are
established via a new family of affine chord Sobolev inequalities. Through
these inequalities, the recent affine isoperimetric inequalities of Haddad
and Ludwig \cite{HL1} are extended from convex bodies to functions, and the
Euclidean chord Sobolev inequalities introduced in \cite{BC} are
strengthened. For compactly supported \(s\)-concave functions with \(s>0\),
sharp monotonicity and reverse monotonicity results for functional radial
mean bodies are obtained, with the equality cases characterized. As
consequences, the corresponding monotonicity theorems for geometric radial
mean bodies due to Gardner and Zhang \cite{GZ} are extended.

\end{abstract}

\maketitle

\section{Introduction}
Affine inequalities, invariant under translations and volume-preserving linear transformations, often strengthen classical Euclidean inequalities. A seminal example is Zhang's affine Sobolev inequality \cite{Zh2}, which significantly improves the classical Sobolev inequality. Wang \cite{Wa} later extended this inequality to functions of bounded variation, yielding a generalized Petty projection inequality (see also \cite{Zh2}), itself a strengthening of the classical isoperimetric inequality. Recently, Milman and Yehudayoff \cite{MY} established affine isoperimetric inequalities for affine quermassintegrals, settling a long-standing conjecture of Lutwak \cite{Lut2}. For further developments on affine inequalities, see, for example, \cite{Bo, Cai, CLYZ, GZ, HS, HJM, HLPRY2, HLPRY,  LYZ1, LYZ2, LYZ4}.

More recently, Haddad and Ludwig \cite{HL1, HL2, HL3} developed a new approach to study affine inequalities. They proved the following sharp affine fractional Sobolev inequalities in \cite{HL1}, which are  stronger than their Euclidean counterparts by Almgren and Lieb \cite{AL} (see also \cite{FS}). 
\vskip 5pt

\noindent{\bf Affine fractional Sobolev inequalities:} For $\alpha\in (-1,0)$ and  $f\in W^{-\alpha,1}(\rn)$,
\begin{equation}\label{aff frac-Sob}
\begin{aligned}
    2\sigma_{n,\alpha}\|f\|_{\frac{n}{n+\alpha}}\leq n\omega_n^{\frac{n-\alpha}{n}}\Big(\frac{1}{n}\int_{\sn}\Big(\int_0^{\infty}&r^{\alpha-1}\|f(\cdot)-f(\cdot+ru)\|_1dr\Big)^{\frac{n}{\alpha}}du\Big)^{\frac{\alpha}{n}}\\&
    \leq \di\frac{|f(x)-f(y)|}{|x-y|^{n-\alpha}}dxdy,
\end{aligned}
\end{equation}
where $\sn$ is the unit sphere in $\rn$ and $\omega_n$ denotes the volume of the unit ball in  $\rn$. 
The sharp constant $\sigma_{n,\alpha}$ is given by
\begin{equation}\label{sigma-alpha}
    \sigma_{n,\alpha}
    =\frac{n 2^{\alpha}\pi^{\frac{n-\alpha-1}{2}}
    \Gamma(\frac{\alpha+1}{2})\Gamma(\frac{n}{2}+1)^{\frac{\alpha}{n}}}
    {|\alpha|\Gamma(\frac{n+\alpha+1}{2})}.
\end{equation}
Here $\Gamma$ denotes the gamma function, $\|f\|_p:=(\int_{\rn}f(x)^pdx)^{1/p}$ denotes the $L^p$ norm of $f$ and $W^{-\alpha,1}(\rn)$ is the space of integrable functions $f$ such that the right-most side of \eqref{aff frac-Sob} is finite.

Eliminating the intermediate term in \eqref{aff frac-Sob} gives the fractional Sobolev inequalities, which converge to the classical Sobolev inequality when $\alpha\rightarrow -1^+$, as shown by Bourgain, Brezis and Mironescu \cite{BBM}. Haddad and Ludwig \cite{HL1} showed that \eqref{aff frac-Sob} becomes the affine Sobolev inequality in the limit $\alpha\rightarrow -1^+$. The fractional Sobolev inequalities have been extensively studied over the past decades (see, e.g., \cite{ACPS, BBM, Da, Kr, Lud1, Lud2}). In  very recent work by the authors \cite{BC}, the following sharp chord Sobolev inequalities for positive $\alpha$ are established.
\vskip 5pt

\noindent{\bf Chord Sobolev inequalities for $\boldsymbol{\alpha>0}$:} 
    For $\alpha\in (0,n)$ and non-negative $f\in L^{\frac{n}{n+\alpha}}(\rn)$, 
    \begin{equation}\label{chord-Sob}
        \sigma_{n,\alpha}\|f\|_{\frac{n}{n+\alpha}}\ge \di\frac{\min\{f(x),f(y)\}}{|x-y|^{n-\alpha}}dxdy.
    \end{equation}
    For $\alpha>n$ and non-zero, non-negative $f\in L^1(\rn)\cap L^{\infty}(\rn)$,
    \begin{equation}\label{chord-Sob>n}
        \sigma_{n,\alpha}\|f\|_1^{\frac{n+\alpha}{n}}\|f\|_{\infty}^{-\frac{\alpha}{n}}\leq \di \frac{\min \{f(x), f(y)\} }{|x-y|^{n-\alpha}}dxdy,
    \end{equation}
    where $\sigma_{n,\alpha}$ is given by \eqref{sigma-alpha}. In both cases, equality holds if and only if $f$ is a constant multiple of the characteristic function of a ball.

In \cite{BC}, fractional Sobolev inequalities, together with \eqref{chord-Sob} and \eqref{chord-Sob>n}, form a complete series called the chord Sobolev inequalities. When $f=\chi_K$ is the characteristic function of a convex body $K\subset\rn$ (a convex compact set with non-empty interior), these inequalities reduce to the isoperimetric inequalities for chord power integrals; see \cite{BC} for details. Chord power integrals have been studied extensively in recent years \cite{BC, Cai2, GXZ, LXYZ2020, XYZZ, XZ}.

One of the main aims of this paper is to prove the following affine chord Sobolev inequalities for $\alpha>0$, which strengthen their Euclidean counterparts \eqref{chord-Sob} and \eqref{chord-Sob>n}. Here $f$ is radially symmetric if its value depends only on $|x|$.

\begin{Thm1}\label{aff chord Sob}
    For $\alpha\in(0,n)$ and non-negative $f\in L^{\frac{n}{n+\alpha}}(\rn)$, 
    \begin{equation}\label{AI-chord}
        \begin{aligned}
        \sigma_{n,\alpha}\|f\|_{\frac{n}{n+\alpha}}\ge n\omega_n^{\frac{n-\alpha}{n}}\Big
        (\frac{1}{n}\int_{\sn}\Big(\int_0^{\infty} r^{\alpha-1}&\big\|\min\{f(\cdot), f(\cdot +ru)\}\big\|_1dr\Big)^{\frac{n}{\alpha}}du\Big)^{\frac{\alpha}{n}}\\&\ge
            \di\frac{\min\{f(x),  f(y)\}}{|x-y|^{n-\alpha}}dxdy.
        \end{aligned}
    \end{equation}
    There is equality in the first inequality if and only if $f$ is a constant multiple of the characteristic function of an ellipsoid. There is equality in the second inequality if $f$ is radially symmetric. 
    \end{Thm1}

\begin{Thm2}\label{theo alpha>n}
For $\alpha>n$ and non-zero, non-negative $f\in L^1(\mathbb{R}^n)\cap L^\infty(\mathbb{R}^n)$, 
    \begin{equation}\label{AI-chord>n}
        \begin{aligned}
        \sigma_{n,\alpha}\|f\|_1^{\frac{n+\alpha}{n}}\|f\|_{\infty}^{-\frac{\alpha}{n}}\le n\omega_n^{\frac{n-\alpha}{n}}\Big
        (\frac{1}{n}\int_{\sn}\Big(\int_0^{\infty} r^{\alpha-1}&\big\|\min\{f(\cdot), f(\cdot +ru)\}\big\|_1dr\Big)^{\frac{n}{\alpha}}du\Big)^{\frac{\alpha}{n}}\\&\le
            \di\frac{\min\{f(x),  f(y)\}}{|x-y|^{n-\alpha}}dxdy.
        \end{aligned}
    \end{equation}
    There is equality in the first inequality if and only if $f$ is a constant multiple of the characteristic function of an ellipsoid. There is equality in the second inequality if $f$ is radially symmetric.

\end{Thm2}

    The affine fractional Sobolev inequalities \eqref{aff frac-Sob}, together with \eqref{AI-chord} and \eqref{AI-chord>n}, are called affine chord Sobolev inequalities. To prove Theorem 1 and Theorem 2, we introduce an anisotropic version of the right-most side of \eqref{AI-chord} (see, e.g., \cite{HL1, Lud1, Lud2}),  which gives rise to a new star-shaped set $\Laf$, defined via its radial function by
\begin{equation*}
    \rho_{\Laf}(u)^{\alpha}=\int_0^{\infty}r^{\alpha-1}\int_{\rn}\min\{f(x),f(x+ru)\}dxdr,~~u\in\sn,
\end{equation*}
(see Section \ref{Laf} for details). The first inequality in Theorem 1 can then be written as
\begin{equation}\label{geo-ai-chord}
    \sigma_{n,\alpha}\|f\|_{\frac{n}{n+\alpha}}\ge n\omega_n^{\frac{n-\alpha}{n}}|\Laf|^{\frac{\alpha}{n}},
\end{equation}
where $|\Laf|$ denotes the volume of $\Laf$. Since both sides of \eqref{geo-ai-chord} are invariant under translations of $f$ and
\[\La(f\circ\phi^{-1})=\phi\Laf\]
for every volume-preserving linear transformation $\phi:\rn\rightarrow\rn$, it follows that inequality \eqref{geo-ai-chord} is affine invariant.

For a convex body $K\subset\rn$, the star-shaped set $\La\chi_K$ is proportional to the radial mean body $\Ra K$ of Gardner and Zhang \cite{GZ}, as will be discussed in Section \ref{Radial mean bodies for log-concave functions}. Haddad and Ludwig \cite{HL1} proved affine isoperimetric inequalities for $\Ra K$ with $\alpha>-1$, stating that

    \begin{equation}\label{HL1-20}
        \begin{aligned}
            \frac{|\Ra K|}{|K|}&\leq \frac{|\Ra B^n|}{|B^n|},&& \alpha\in(-1,n),\\[0.5em]
            \frac{|\Ra K|}{|K|}&\geq \frac{|\Ra B^n|}{|B^n|},&& \alpha>n.
        \end{aligned}
    \end{equation}
    Equality holds if and only if $K$ is an ellipsoid.

In \cite{BC}, the authors gave the following level-set formulation of the radial mean body $\Raf$ for log-concave functions \(f\in L^1(\rn)\):
\begin{equation}\label{Raf-intro}
    \rho_{\Raf}(u)^{\alpha}=\int_0^{\|f\|_{\infty}}\rho_{\Ra \{f\ge t\} }(u)^{\alpha}d\mu_f(t),~\alpha\neq0,
\end{equation}
and
\begin{equation*}
    \log\rho_{\rR_0f}(u)=\int_0^{\|f\|_\infty}\log\rho_{\rR_0 \{f\ge t\} }(u)d\mu_f(t),
\end{equation*}
where $d\mu_f(t)=\frac{|\{f\ge t\}|}{\|f\|_1}dt$ is a probability measure on $\tr$. 

We note that Haddad and Ludwig \cite{HL3} also use the notation $\Raf$ for the $\alpha$-th radial mean body, but with a different definition. The present definition can be viewed as an $L^1$ extension to their definitions. We also refer the reader to \cite{LSU}, where functional higher order radial mean bodies were first introduced by a different construction; in particular,
the corresponding \(L^1\) case coincides with the body defined above.

It was shown in \cite{BC} that $\Raf$ is proportional to $\Laf$ for $\alpha>0$, while for $\alpha\in(-1,0)$ it is proportional to the fractional polar projection body $\Pi^{*,-\alpha}f$ arising from the affine fractional Sobolev inequalities \eqref{aff frac-Sob}. Consequently, affine chord Sobolev inequalities yield the following affine isoperimetric inequalities for $\Raf$: for a log-concave function $f\in L^{1}(\rn)$,
\begin{equation}\label{AI-Raf}
\begin{aligned}
\frac{|\Ra f|}{\|f\|_{\frac{1}{2}}\|f\|_1^{-1}} &\leq \frac{|\Ra B^n|}{|B^n|}, && \alpha\in(-1,n),\\[0.5em]
\frac{|\Ra f|}{\|f\|_1\|f\|_{\infty}^{-1}} &\geq \frac{|\Ra B^n|}{|B^n|}, && \alpha>n,
\end{aligned}
\end{equation}
with equality if and only if $f$ is a constant multiple of the characteristic function of an ellipsoid. In particular, when $f=\chi_K$ is the characteristic function of a convex body, \eqref{AI-Raf} reduces to the geometric inequality \eqref{HL1-20}.

Of particular interest, for log-concave $f\in L^1(\rn)$, Haddad and Ludwig \cite{HL1} proved
\[
\lim_{\alpha\rightarrow -1^+}(1+\alpha)\|f\|_1\rho_{\Raf}(u)^{\alpha}=\rho_{\Pi^*f}(u)^{-1},
\]
where $\Pi^*f$ is the polar projection body of $f$ introduced by Zhang \cite{Zh2} and extended by Wang \cite{Wa} (see Section \ref{frac ppb}). For the other endpoint $\alpha\to\infty$, the monotonicity of $L^p$ norms with respect to $p$, together with \eqref{Raf-intro}, leads to the definition of $\rR_{\infty}f$: 
\[
\rR_{\infty} f = \rR_{\infty}\,\supp(f) = {\rm D}\,\supp(f),
\]
where ${\rm D}K := \{x-y: x,y\in K\}$ denotes the difference body, and we set $\rR_\infty f = \rn$ if $\supp(f)=\rn$. In Section \ref{Radial mean bodies for log-concave functions}, we prove the following monotonicity for log-concave $f\in L^1(\rn)$:
\begin{equation}\label{increasing}
    \|f\|_{\infty}\Pi^*f\subset \Big((1+\alpha)\frac{\|f\|_1}{\|f\|_{\infty}}\Big)^{1/\alpha}\Raf\subset \Big((1+\beta)\frac{\|f\|_1}{\|f\|_{\infty}}\Big)^{1/\beta}\rR_{\beta} f\subset \rR_{\infty} f,
\end{equation}
for $-1<\alpha<\beta<\infty$. When $f=\chi_K$, these reduce to the geometric case of Gardner and Zhang~\cite{GZ}.

Gardner and Zhang \cite{GZ} also established the reverse inclusion for radial mean bodies, which is more involved. For a convex body $K\subset\rn$, they proved
\begin{equation}\label{GZ-mian5.5}
    {\rm D}K\subset c_{n,\beta}\rR_{\beta} K\subset c_{n,\alpha}\Ra K\subset n|K|\Pi^* K,
\end{equation}
for $-1<\alpha<\beta<\infty$, where $c_{n,\alpha}=(nB(\alpha+1,n))^{-1/\alpha}$ and $B(\cdot,\cdot)$ is the Beta function. Equality holds in each inclusion if and only if $K$ is a simplex. When $\beta=n$, the left-most inclusion gives the Rogers--Shephard inequality \cite{RS}, while for $\alpha=n$, the right-most inclusion yields Zhang's inequality \cite{Zh1} (see also Ball \cite{ball}).

Building on the functional radial mean body $\Raf$ and the monotonicity result \eqref{increasing}, we aim to extend the reverse inclusion \eqref{GZ-mian5.5} from sets to functions, stated as below.
\vskip 5pt

\noindent{\bf Question.} For a non-zero, log-concave function $f \in L^1(\mathbb{R}^n)$ with compact support, does
\[\rR_{\infty}f\subset c_{n,\beta}\rR_{\beta}f\subset c_{n,\alpha}\Raf\subset n\|f\|_1\Pi^*f\]
hold for $-1<\alpha<\beta<\infty$, with equality if and only if $f=c\chi_{\triangle}$? Here $c_{n,\alpha}$ is the constant in \eqref{GZ-mian5.5}, $c>0$ and $\triangle$ denotes a simplex in $\rn$.

Unfortunately, this inclusion fails for general log-concave functions, as we show by a counterexample in Section \ref{Reversed inequalities}. This failure leads us  to consider $s$-concave functions with $s>0$. 

Langharst, Mar\'in Sola and Ulivelli \cite{LSU} obtained related inclusion results through a different Ball-body framework; see also the discussion in Section \ref{Reversed inequalities}. Our approach is based on the level-set framework and gives a direct proof of the inclusions \eqref{main3-reversed inclusion} considered here, together with the equality characterization stated in Theorem 3.

\begin{Thm3}
Let $s>0$ and let $f\in L^1(\mathbb{R}^n)$ be a non-zero, non-negative, $s$-concave function with compact support.  
For $-1<\alpha<\beta<\infty$, the following chain of inclusions holds:
\begin{equation}\label{main3-reversed inclusion}
    \rR_{\infty} f
    \subset
    c_{n,\beta}(s)\,\rR_{\beta} f
    \subset
    c_{n,\alpha}(s)\,\rR_{\alpha} f
    \subset
    \Big(n+\frac{1}{s}\Big)\|f\|_{1}\,\Pi^{*}f,
\end{equation}
where $c_{n,\alpha}(s)=\bigl((n+\tfrac{1}{s})\,B(\alpha+1,n+\tfrac{1}{s})\bigr)^{-1/\alpha}$. If $f$ is upper semicontinuous, there is equality in each inclusion if and only if $f(x)=a(1-\|x-x_0\|_{\triangle-x_0})_+^{1/s}$, where $x_0$ is the maximizer of $f$ with $a=f(x_0)=\|f\|_{\infty}$, and $\triangle$ is a simplex in $\rn$ containing $x_0$.
\end{Thm3}

 Letting $\beta=n$ and $\alpha=n$ in \eqref{main3-reversed inclusion} yields the Rogers-Shephard inequality and Zhang inequality, respectively, for $s$-concave functions; see Section \ref{Reversed inequalities} for detailed statement.

 The optimizer in \eqref{main3-reversed inclusion} admits a simple geometric interpretation. Let $\triangle$ be a simplex in $\rn$ with vertices $v_0,\ldots, v_n$, and identify $\rn$ with the hyperplane $\rn\times\{0\}\subset \tr^{n+1}$. For 
$f^s(x)=a(1-\|x-x_0\|_{\triangle-x_0})_+ \quad \text{with } x_0\in \triangle$,  
its hypograph is given by the set
\[
\{(x,l)\in \tr^{n+1}: x\in \supp(f^s),~0\le l \le f^s(x)\}
\]
and coincides exactly with the simplex in $\tr^{n+1}$ having vertices $v_0,\ldots, v_n, (x_0,a)$. The equality characterization of \eqref{main3-reversed inclusion} relies on the equality case of the Borell–Brascamp–Lieb inequality. Its highly rigid structure accounts for the difficulty in the study of the equality cases.

\section{Preliminaries}

In this section, we collect basic notation and results on function spaces and Schwarz symmetrals, together with some basic geometric facts. Useful general references include the books by Gardner \cite{Gar}, Schneider \cite{Sch}, and Schneider and Weil \cite{SW}.

\subsection{\texorpdfstring{$L^p$}{} spaces}\hspace{\fill}

Let $p>-1$ and $(X,\mathcal{B},\mu)$ be a measure space. In this work, we primarily focus  on the cases $X=\rn$ or $X\subset\rn$ with $\mu$ being the $n$-dimensional Lebesgue measure. 

For $p\neq 0$ and a measurable function $f:X\rightarrow \tr$, let
\[\|f\|_p=\Big(\int_{X}|f(x)|^p d\mu(x)\Big)^{1/p},\]
and
\[\|f\|_{\infty}=\inf\{C>0: |f(x)|\leq C~~\text{almost everywhere on}~X\}.\]
In particular, for $p\in(-1,0)$, we set $\|f\|_p = \infty$ if $\mu(\{x\in X: f(x)=0\})>0$.  The associated space $L^p(X)$ is then defined as
\[L^p(X)=\{f:X\rightarrow \tr : f~\text{is measurable},~\|f\|_p<\infty\},\]
where  $f=g$ in $L^p(X)$ means that they are equal  almost everywhere.

When $\mu(X)=1$, the $L^p$ norm is monotone in $p$: for $-1<p<q<0$ or $0<p<q<\infty$, 
\[\|f\|_p \le \|f\|_q,
\]
as a consequence of Jensen's inequality. In fact, this monotonicity also extends for $p<0<q$. To see this, under the assumption $\mu(X)=1$, we first note the limit
\[\lim_{p\rightarrow 0}\log\|f\|_p=\int_X\log |f(x)|d\mu(x),\]
which follows from the continuity of $\log\|f\|_p$ with respect to $p$. Applying Jensen's inequality then gives
\[\|f\|_p\leq \exp\Big(\int_X\log|f(x)|d\mu(x)\Big)\leq \|f\|_q,\]
for $p<0<q$.

\subsection{Symmetrization}\hspace{\fill}

For a Borel set $E\subset\rn$ with finite measure, the Schwarz symmetral of $E$, denoted by $E^{\star}$, is the  closed, centered ball having the same volume as $E$.  Let $f:\rn\rightarrow [0,\infty)$ be a measurable function. For 
$t>0$, the superlevel set of $f$ is defined by $\{f\ge t\}=\{x\in \rn: f(x)\ge t\}$. We say $f$ is non-zero if $\{f\neq 0\}$ has positive measure. 

If the superlevel sets of $f$ have finite measure, the layer cake formula gives 
  \[f(x)=\int_0^{\infty}\chi_{\{f\ge t\}}(x)dt\]
for almost every $x\in\rn$.  The Schwarz symmetral of $f$, denoted by $f^{\star}$, is defined as
\[f^{\star}(x)=\int_0^{\infty}\chi_{\{f\ge t\}^{\star}}(x)dt\]
for $x\in\rn$, so that $f^\star$ is radially symmetric and the superlevel sets satisfy  
$$|\{f^\star \ge t\}| = |\{f \ge t\}|$$ 
for all $t>0$.


A central tool in our proofs is the Riesz rearrangement inequality (see, for example  \cite[Theorem 3.7]{LL}) together with the characterization of equality by Burchard  \cite{Bu}.

\begin{thm}[Riesz's rearrangement inequality]\label{rri}
    For measurable $f,g,k:\rn\rightarrow [0, \infty)$ whose superlevel sets have finite measure.
    \[\di f(x)k(x-y)g(y)dxdy\leq \di f^{\star}(x)k^{\star}(x-y)g^{\star}(y)dxdy.\]
\end{thm}

\begin{thm}[Burchard]\label{rri-cha}
    Let $A, B$ and $C$ be sets of finite positive measure in $\rn$ and denote by $a, b$ and $c$ the radii of their Schwarz symmetrals $A^{\star}, B^{\star}$ and $C^{\star}$. For $|a-b|<c<a+b $, there is equality in
    \[\di\chi_A(x)\chi_B(x-y)\chi_C(y)dxdy\leq \di\chi_{A^{\star}}(x)\chi_{B^{\star}}(x-y)\chi_{C^{\star}}(y)dxdy\]
    if and only if, up to sets of measure zero,
    \[A=x_0+a D,~B=x_1+b D,~C=x_2+c D,\]
    where $D$ is a centered ellipsoid, and $x_0,x_1$ and $x_2=x_0+x_1$ are vectors in $\rn$.
\end{thm}

\subsection{Star-shaped sets and dual mixed volumes}\label{dual mixed}\hspace{\fill}

A set $K \subset \mathbb{R}^n$ is called  star-shaped  (with respect to the origin) if for every $x\in K$, the line segment $[o,x]$ is contained in $K$. For a star-shaped set $K$ containing the origin in its interior, the gauge function $\|\cdot\|_K:\rn\rightarrow [0,\infty]$ is defined by
\begin{equation}\label{gauge}
    \|x\|_K=\inf\{\lambda >0: x\in \lambda K\}.
\end{equation}
If $K$ contains the origin, possibly on its boundary, then for any point $x$, the gauge function $\|x\|_K$ is defined by \eqref{gauge} if there exists $r>0$ such that $rx\in K$; otherwise, $\|x\|_K=\infty$.

The corresponding radial function $\rho_K:\rn\backslash\{0\}\rightarrow [0,\infty]$ is
\[\rho_{K}(x)=\|x\|_K^{-1}=\sup\{r\ge 0: rx\in K\}.\]
In particular, if $K$ is the unit Euclidean ball, $\|\cdot\|_K$ coincides with the Euclidean norm $|\cdot|$. 

For a star-shaped set $K$ with measurable radial function, the volume, or the $n$-dimensional Lebesgue measure of $K$ is given by
\[|K|=\frac{1}{n}\int_{\sn}\rho_{K}(u)^ndu,\]
where $du$ denotes the spherical Lebesgue measure. Following Lutwak \cite{Lut1}, the dual mixed volume for star-shaped sets $K$ and $L$ is defined by
\[\tilde{V}_{\alpha}(K,L)=\frac{1}{n}\int_{\sn}\rho_{K}(u)^{n-\alpha}\rho_L(u)^{\alpha}du,\]
for $\alpha\in\tr\backslash\{0,n\}$. In particular, 
\[\tilde{V}_{\log}(K,L)=\lim_{\alpha\rightarrow 0}\log\Big(\frac{\tilde{V}_{\alpha}(K,L)}{|K|}\Big)^{1/\alpha}=\frac{1}{n|K|}\int_{\sn}\rho_K(u)^n\log\Big(\frac{\rho_L(u)}{\rho_K(u)}\Big)du.\]

If  $\alpha\in(0,n)$ and  $K,L$ are star-shaped sets with finite volume, the dual mixed volume inequality states that
\begin{equation}\label{dmvi1}
    \tilde{V}_{\alpha}(K,L)\leq |K|^{\frac{n-\alpha}{n}}|L|^{\frac{\alpha}{n}},
\end{equation}
where equality holds if and only if $K$ and $L$ are dilates, i.e., $\rho_K(u)=c\rho_L(u)$ for some $c\ge0$ and almost all $u\in\sn$. 
For $\alpha<0$ or $\alpha>n$, the dual mixed volume inequality states that
\begin{equation}\label{dmvi2}
    \tilde{V}_{\alpha}(K,L)\geq |K|^{\frac{n-\alpha}{n}}|L|^{\frac{\alpha}{n}},
\end{equation}
where equality holds if and only if $K$ and $L$ are dilates.

\subsection{Log-concave and \texorpdfstring{$s$}{}-concave functions}\label{log s}\hspace{\fill}

A function $f:\rn\rightarrow[0,\infty)$ is  log-concave if $\log f$ is a concave function on $\rn$, with values in $[-\infty,\infty)$. For $s>0$, a function $f:\rn\rightarrow[0,\infty)$ is $s$-concave if $f^s$ is concave on its support.

\begin{lem}\cite[Lemma 2.1]{ABG}\label{abg2.1}
    Let $f:\rn\rightarrow[0,\infty)$ be a log-concave and integrable function. Then the function $g:\rn\rightarrow [0,\infty)$ defined by
    \[g(x)=\int_{\rn}\min\{f(y), f(y+x)\}dy\]
    is even, log-concave, with $\|g\|_{\infty}=g(o)=\|f\|_1$ and $\supp(g)$ containing the origin point in its interior.
\end{lem}

For $\lambda\in[0,1]$, $a,b>0$ and $s>0$, we denote by
\begin{equation*}
            {\mathcal M}_{\lambda}^{s}(a,b)=\left\{
                \begin{aligned}
                &((1-\lambda)a^s+\lambda b^s)^{1/s}&&\text{for}~~ab\neq 0,\\
                &0&&\text{for}~~ab=0.\\
                \end{aligned}\right.
        \end{equation*}
For an $s$-concave function $f$, we have $f((1-\lambda)x+\lambda y)\ge \mathcal{M}_{\lambda}^s(f(x), f(y))$ for $\lambda\in[0,1]$. We say $f$ is $(\lambda_0, s)$-concave, if $f((1-\lambda_0)x+\lambda_0 y)\ge \mathcal{M}_{\lambda_0}^s(f(x), f(y))$ for fixed $\lambda_0\in(0,1)$.

The following theorem of Balogh and Krist\'aly \cite{BK} gives the equality characterization of the Borell-Brascamp-Lieb inequality.

\begin{thm}\cite[Theorem 3.1]{BK}\label{BBL-eq}
    Let $\lambda\in(0,1)$ and $s>-\frac{1}{n}$. Let $F,G,H:\rn\rightarrow [0,\infty)$ be  non-zero, compactly supported integrable functions satisfying
    \[H((1-\lambda)x+\lambda y)\ge {\mathcal M}_{\lambda}^s(F(x),G(y)).\]
    Then there is equality in
    \begin{equation}\label{BBL}
    \int_{\rn} H \ge {\mathcal M}_{\lambda}^{\frac{s}{ns+1}}\Big(\int_{\rn} F, \int_{\rn} G\Big),
    \end{equation}
    if and only if there is $x_0\in\rn$, an $(\lambda_0,s)$-concave function $\Phi: \supp(F)\rightarrow [0,\infty)$ with $\lambda_0=\frac{\lambda c_0}{1-\lambda+\lambda c_0}$ and $c_0=(\frac{|\supp (G)|}{|\supp(F)|})^{1/n}$ such that up to null measure sets,
    \[\supp (G)=c_0\,\supp(F)+x_0,~~\text{and}~~\supp(H)=(1-\lambda+\lambda c_0)\,\supp (F)+\lambda x_0,\]
    and for almost all $x\in\supp(F)$,
    \begin{equation*}
        \left\{
                \begin{aligned}
                &F(x)=\Phi(x)&&,\\
                &G(c_0x+x_0)=c_0^{\frac{1}{s}}\Phi(x)&&,\\
                &H((1-\lambda+\lambda c_0)x+\lambda x_0)=\Big[\mathcal{M}_{\lambda}^{\frac{s}{ns+1 }}\Big(1, c_0^{\frac{ns+1}{s}}\Big)\Big]^{\frac{1}{ns+1}}\Phi(x)&&.
                \end{aligned}\right.
    \end{equation*}
\end{thm}

\subsection{Fractional polar projection bodies}\label{frac ppb}\hspace{\fill}
\vskip 7pt
For $\alpha\in(-1,0)$ and measurable $f:\rn\rightarrow\tr$, the fractional polar projection body $\Pi^{*,-\alpha}f$ is defined by
\begin{equation*}
    \rho_{\Pi^{*,-\alpha}f}(u)^{\alpha}=\int_0^{\infty}r^{\alpha-1}\int_{\rn}|f(x)-f(x+ru)|dxdr.
\end{equation*}
The affine fractional Sobolev inequalities \eqref{aff frac-Sob} state that
\begin{equation*}
    2\sigma_{n,\alpha}\|f\|_{\frac{n}{n+\alpha}}\leq n\omega_n^{\frac{n-\alpha}{n}}|\Pi^{*,-\alpha}f|^{\frac{\alpha}{n}}\leq \di \frac{|f(x)-f(y)|}{|x-y|^{n-\alpha}}dxdy
\end{equation*}
for $f\in W^{-\alpha,1}(\rn)$

Let $\mathcal{B}(\rn)$ denote the class of Borel sets in $\rn$. Let $C_c^{\infty}(\rn; \rn)$ denote the set of smooth vector fields with compact support. Moreover, ${\rm div}\, \nu$ denotes the divergence of $\nu\in C_c^{\infty }(\rn; \rn)$. We say that $f\in L^1(\rn)$ is a function of bounded variation on $\rn$ if there is a finite vector-valued Radon measure $\mathcal{D}f: \mathcal{B}(\rn)\rightarrow \rn$ such that
\[\int_{\rn}\nu(x)\cdot n_f(x)d|\mathcal{D}f|(x)=-\int_{\rn}f(x){\rm div}\,\nu(x)dx\]
for every $\nu\in C_c^{\infty}(\rn; \rn)$, where $|\mathcal{D}f|: \mathcal{B}(\rn)\rightarrow [0,\infty)$ denotes the variation measure of $\mathcal{D} f$ and $n_f: \rn\rightarrow \rn$ the Radon-Nikodym derivative of $\mathcal{D} f$ with respect to $|\mathcal{D}f|$.

For $f\in L^1(\rn)$ with bounded variation, the polar projection body $\Pi^*f$ is defined by
\[\|u\|_{\Pi^{*}f}=\frac{1}{2}\int_{\rn}|n_f(x)\cdot u|\,d|\mathcal{D}f|(x).\]
For more information about $\Pi^* f$, we refer the reader to \cite{Wa}.

\section{The star-shaped set \texorpdfstring{$\La f$}{}}\label{Laf}

Building on the approach of Haddad and Ludwig \cite{HL1,HL2, HL3}, we introduce the following anisotropic integral. Let $f:\rn\rightarrow[0,\infty)$ be a measurable function.  We consider
\[\di\frac{\min\{f(x), f(y)\}}{\|x-y\|_K^{n-\alpha}}dxdy,\]
where $K\subset\rn$ is a star-shaped set with measurable radial function and $\alpha>0$.

Here,  we first recall the definition of $\La f$: 
\[\rho_{\Laf}(u)^{\alpha}=\int_0^{\infty}r^{\alpha-1}\int_{\rn}\min\{f(x), f(x+ru)\}dxdr.\]
Using Fubini's theorem and polar coordinates, we obtain that
\begin{equation*}
    \begin{aligned}
    \di\frac{\min\{f(x), f(y)\}}{\|x-y\|_K^{n-\alpha}}dxdy&=\di\frac{\min\{f(y+z), f(y)\}}{\|z\|_K^{n-\alpha}}dzdy\\
    &=\int_{\rn}\int_{\sn}\int_0^{\infty}\rho_K(ru)^{n-\alpha}r^{n-1}\min\{f(y), f(y+ru)\}drdudy\\
    &=\int_{\sn}\rho_K(u)^{n-\alpha}\int_0^{\infty}r^{\alpha-1}\int_{\rn}\min\{f(y), f(y+ru)\}dydrdu\\
    &=\int_{\sn}\rho_K(u)^{n-\alpha}\rho_{\Laf}(u)^{\alpha}du.
    \end{aligned}
\end{equation*}
That is,
\begin{equation}\label{a-i}
    \di\frac{\min\{f(x), f(y)\}}{\|x-y\|_K^{n-\alpha}}dxdy=n\tilde{V}_{\alpha}(K,\Laf).
\end{equation}
In particular,
\begin{equation*}
    \begin{aligned}
   |{\rm L}_n f|&=\frac{1}{n}\int_{\sn}\rho_{{\rm L}_n f}(u)^ndu\\
          &=\frac{1}{n}\int_{\sn}\int_0^{\infty}\int_{\rn}r^{n-1}\min\{f(x), f(x+ru)\}dxdrdu\\ 
 &=\frac{1}{n}\int_{\rn}\int_{\rn}\min\{f(x), f(y)\}dxdy.\\
\end{aligned}
\end{equation*}

Let $f:\rn\rightarrow [0,\infty)$ be a measurable function. Then the following hold:
\vskip 6pt

{For $\alpha\in(0,n)$}, the dual mixed volume inequality \eqref{dmvi1} together with \eqref{a-i} implies
\[\sup\Big\{\di \frac{\min\{f(x), f(y)\}}{\|x-y\|_K^{n-\alpha}}dxdy: K\subset\rn ~\text{star-shaped}, |K|=\omega_n\Big\}=n\omega_n^{\frac{n-\alpha}{n}}|\Laf|^{\frac{\alpha}{n}}.\]
For $\alpha>n$, the dual mixed volume inequality \eqref{dmvi2} together with \eqref{a-i} yields
\[\inf\Big\{\di \frac{\min\{f(x), f(y)\}}{\|x-y\|_K^{n-\alpha}}dxdy: K\subset\rn ~\text{star-shaped}, |K|=\omega_n\Big\}=n\omega_n^{\frac{n-\alpha}{n}}|\Laf|^{\frac{\alpha}{n}}.\]
Moreover, in either regime a suitable dilate of $\Laf$ attains the supremum (respectively, infimum), provided $|\Laf|$ is finite.

We conclude this section with a result on the convexity of $\Laf$ when $f$ is log-concave.

\begin{proposition}
Let $\alpha>0$. If $f:\mathbb{R}^n\to [0,\infty)$ is a log-concave, integrable function with compact support, then $\Laf$ is a convex body.   
\end{proposition}
\begin{proof}
For $u\in\sn$, we define
\begin{align*}
    g(ru)=\int_{\mathbb{R}^n}\min\{f(x),f(x+ru)\}d x.
\end{align*}
Since $f$ is log-concave, Lemma \ref{abg2.1} implies that $g$ is also log-concave.

Note that
\begin{align*}
    \rho_{\Laf}(u)=\Big(\int_{0}^{\infty}r^{\alpha-1} g(ru)dr\Big)^{1/\alpha}.
\end{align*}
It follows from \cite[Corollary 4.2]{GZ} that ${\Laf}$ is a convex body.
\end{proof} 

\section{Affine chord Sobolev inequalities}\label{ACS}

In this section, we establish the affine chord Sobolev inequalities stated in Theorem 1 and Theorem 2. To this end, we first prove a crucial lemma based on the Riesz rearrangement inequality.

Recall that a set $A\subset \mathbb{R}^n$ is homothetic to $B\subset \mathbb{R}^n$ if there exist $x\in \mathbb{R}^n$ and $a>0$ such that $A = x + aB$. Moreover, $A$ is said to be \emph{equivalent} to $B$ if their symmetric difference $A\triangle B$ is a null set. The following lemma provides a general form of the Riesz rearrangement inequality; see also \cite{Cai, HL1, HL2, HL3}.

\begin{lem}\label{rri-min}
    Let $q>0$ and $K\subset\rn$ be a star-shaped set with measurable radial function and $|K|>0$. For a measurable function $f:\rn\rightarrow[0,\infty)$ such that
    \[\di\frac{\min\{f(x), f(y)\}}{\|x-y\|_K^q}dxdy<\infty,\]
    there is
    \[\di\frac{\min\{f(x), f(y)\}}{\|x-y\|_K^q}dxdy\leq \di\frac{\min\{f^{\star}(x), f^{\star}(y)\}}{\|x-y\|_{K^{\star}}^q}dxdy.\]
    Equality holds {if and only if} $K$ is {equivalent to} a centered ellipsoid $D$ and for almost all $t>0$, the level set $\{f\ge t\}$ has measure zero or is homothetic to $D$ up to null sets.
\end{lem}

\begin{proof}
    For $z\in\mathbb{R}^n\backslash\{o\}$,
    \[\|z\|_K^{-q}=\int_0^{\infty}\chi_{r^{-1/q}K}(z)dr.\]
    Using the Fubini's theorem, we have
    \begin{align*}
        \di\frac{\min\{f(x), f(y)\}}{\|x-y\|_K^{q}}dxdy&=\int_{0}^{\infty}\di\frac{\chi_{\{f\ge t\}}(x)\chi_{\{f\ge t\}}(y)}{\|x-y\|_K^{q}}dxdydt\\
        &=\int_0^{\infty}\int_0^{\infty}\di\chi_{\{f\ge t\}}(x)\chi_{r^{-1/q}K}(x-y)\chi_{\{f\ge t\}}(y)dxdydrdt.
    \end{align*}
    By the Riesz rearrangement inequality, Theorem \ref{rri-cha}, we have
    \begin{equation}\label{anrri-0n}
        \begin{aligned}
        \di\chi_{\{f\ge t\}}(x)\chi_{r^{-1/q}K}&(x-y)\chi_{\{f\ge t\}}(y)dxdy\\
        &\leq\di\chi_{\{f^{\star}\ge t\}}(x)\chi_{r^{-1/q}K^{\star}}(x-y)\chi_{\{f^{\star}\ge t\}}(y)dxdy,
    \end{aligned}
    \end{equation}   
    which implies that
    \[\di\frac{\min\{f(x), f(y)\}}{\|x-y\|_K^{q}}dxdy\leq\di\frac{\min\{f^{\star}(x), f^{\star}(y)\}}{\|x-y\|_{K^{\star}}^{q}}dxdy.\]

    If equality holds, then \eqref{anrri-0n} is an equality for almost all $(t,r)\in (0,\infty)^2$. For sufficiently large $r>0$, the assumptions of Theorem \ref{rri-cha} are satisfied. Consequently, up to sets of measure zero, we have
\[\{f\ge t\} = x + a D, \qquad r^{-1/q} K = b D,\]
where $a,b>0$, $x\in \rn$, and $D$ is a centered ellipsoid with the same volume as $K$. It follows that 
$K = \big(\frac{|K|}{|D|}\big)^{1/n} D$, showing that $D$ is independent of $t$ and $r$, which completes the proof.
\end{proof}

The next lemma shows that, for $\alpha\in(0,n)$, the volume of $\Laf$ does not increase under Schwarz symmetrals.

\begin{lem}\label{rri-v}
    Let $\alpha\in(0,n)$. For non-negative $f\in L^{\frac{n}{n+\alpha}}(\rn)$,
    \[|\Laf|\leq |\Laf^{\star}|.\]
    If $|\Laf^{\star}|<\infty$, there is equality if and only if the level set $\{f\ge t\}$ has measure zero or is homothetic to an ellipsoid for almost all $t>0$, up to null sets.
\end{lem}

\begin{proof}
    We first assume that $|\Laf|<\infty$. By letting $q=n-\alpha$  in Lemma~\ref{rri-min} and combining it with \eqref{a-i}, we obtain
    \[\tilde{V}_{\alpha}(K,\Laf)\leq \tilde{V}_{\alpha}(K^{\star}, \Laf^{\star}).\]
    With $K=\Laf$ and the dual mixed volume inequality \eqref{dmvi1}, we obtain
    \begin{align*}
    |\Laf|&=\tilde{V}_{\alpha}(\Laf, \Laf)\\
         &\leq \tilde{V}_{\alpha}((\Laf)^{\star}, \Laf^{\star})\\
         &\leq |(\Laf)^{\star}|^{\frac{n-\alpha}{n}}|\Laf^{\star}|^{\frac{\alpha}{n}}\\
         &=|\Laf|^{\frac{n-\alpha}{n}}|\Laf^{\star}|^{\frac{\alpha}{n}},
    \end{align*}
    which implies that
    \[|\Laf|\leq |\Laf^{\star}|.\]
    The equality case follows from Lemma \ref{rri-min}.

    If $|\Laf|=\infty$, define
    \[f_k(x)=f(x)\chi_{kB^n}(x),~~k\in\mathbb N.\]
    By the monotone convergence theorem,
    \[\lim_{k\rightarrow\infty}\int_0^{\infty}r^{\alpha-1}\int_{\rn}\min\{f_k(x), f_k(x+ru)\}dxdr=\int_0^{\infty}r^{\alpha-1}\int_{\rn}\min\{f(x), f(x+ru)\}dxdr,\]
    with the convergence being monotone increasing. Applying the monotone convergence theorem once more then gives
    \begin{align*}
        \lim_{k\rightarrow\infty}\int_{\sn}\Big(\int_0^{\infty}r^{\alpha-1}\int_{\rn}\min\{&f_k(x),f_k(x+ru)\}dxdr\Big)^{\frac{n}{\alpha}}du\\
        &=\int_{\sn}\Big(\int_0^{\infty}r^{\alpha-1}\int_{\rn}\min\{f(x), f(x+ru)\}dxdr\Big)^{\frac{n}{\alpha}}du.
    \end{align*}
    That is,
    \begin{equation}\label{lim-k}
        \lim_{k\rightarrow\infty}|\La f_k|=|\Laf|=\infty.
    \end{equation}

    Since $f\in L^{\frac{n}{n+\alpha}}(\rn)$ and $f_k$ has compact support, $|\La f_k|<\infty$. By the definition of the Schwarz symmetral, we have $(f_k)^{\star}\leq f^{\star}$. Then the first part of the proof implies that
    \[|\La f_k|\leq |\La(f_k)^{\star}|\leq |\La f^{\star}|.\]
    Taking the limit as $k\rightarrow \infty$ in the previous inequality, we obtain from \eqref{lim-k} that
    \[|\La f^{\star}|=\infty,\]
    which gives the desired result.
\end{proof}

\begin{proof}[Proof of Theorem 1]
    By \eqref{a-i} and the dual mixed volume inequality \eqref{dmvi1}, we have
    \begin{equation*}
        \di\frac{\min\{f(x),f(y)\}}{|x-y|^{n-\alpha}}dxdy=n\tilde{V}_{\alpha}(B^n, \Laf)\leq n\omega_n^{\frac{n-\alpha}{n}}|\Laf|^{\frac{\alpha}{n}}.
    \end{equation*}
    There is equality precisely when $\Laf$ is a ball, which is the case for radially symmetric functions.
    
    For the second inequality, noting that $\|f\|_{\frac{n}{n+\alpha}}=\|f^{\star}\|_{\frac{n}{n+\alpha}}$ and that $\La f^{\star}$ is a ball, we obtain from the chord Sobolev inequality \eqref{chord-Sob}, \eqref{a-i}, and Lemma \ref{rri-v} that
    \begin{align*}
        \sigma_{n,\alpha}\|f\|_{\frac{n}{n+\alpha}}&\ge \di\frac{\min\{f^{\star}(x), f^{\star}(y)\}}{|x-y|^{n-\alpha}}dxdy\\
        &=n\tilde{V}_{\alpha}(B^n, \La f^{\star})\\
        &=n\omega_n^{\frac{n-\alpha}{n}}|\La f^{\star}|^{\frac{\alpha}{n}}\\
        &\ge n\omega_n^{\frac{n-\alpha}{n}}|\Laf|^{\frac{\alpha}{n}}.
    \end{align*}
    If equality holds throughout, it follows from the chord Sobolev inequality  that $f^{\star}=c\chi_B$, where $B$ is a ball. By Lemma \ref{rri-v}, we have $f=c\chi_D$, where $D$ is an ellipsoid.
\end{proof}

For $\alpha>n$, we state the following analogue of Lemma \ref{rri-min}, which also generalizes \cite[Lemma 3.2]{BC}. The proof follows a similar approach and is included for completeness.

\begin{lem}\label{q>n}
    Let $q>0$ and $K\subset\rn$ be a star-shaped set with measurable radial function and $|K|>0$. If  $f:\rn\rightarrow [0,\infty)$ is an integrable function such that
    \[\di\min\{f(x), f(y)\}\|x-y\|_K^q dxdy<\infty,\]
    then
    \[\di\min\{f(x), f(y)\}\|x-y\|_K^qdxdy \geq \di\min\{f^{\star}(x), f^{\star}(y)\}\|x-y\|_{K^{\star}}^qdxdy.\]
    Equality holds if and only if $K$ is equivalent to a centered ellipsoid $D$ and for almost all  $t>0$, the level set $\{f\ge t\}$ has measure zero or is homothetic to $D$ up to sets of measure zero.
\end{lem}

\begin{proof}
For $z\in\mathbb{R}^n\backslash\{o\}$, 
\begin{align*}
    \|z\|_K^q=\int_0^\infty \chi_{\mathbb{R}^n\setminus r^{1/q}K}(z)dr.
\end{align*}
Using  Fubini's theorem, we obtain
\begin{equation}\label{cal-adi}
\begin{aligned}
\int_{\mathbb{R}^n}  \int_{\mathbb{R}^n} \min\{f(x),&f(y)\}\|x-y\|_K^qdxdy\\
&= \int_0^\infty  \int_{\rn}  \int_{\rn} \chi_{\{f\ge t\}}(x)\|x-y\|_K^q \chi_{\{f\ge t\}}(y)dxdydt\\
 &= \int_0^\infty  \int_0^\infty \int_{\rn}  \int_{\rn}\chi_{\{f\ge t\}}(x)\chi_{\mathbb{R}^n\setminus r^{1/q}K}(x-y) \chi_{\{f\ge t\}}(y)dxdydtdr\\
 &=\int_0^\infty  \int_0^\infty \int_{\rn}  \int_{\rn}\chi_{\{f\ge t\}}(x)(1-\chi_{r^{1/q}K}(x-y)) \chi_{\{f\ge t\}}(y)dxdydtdr.
\end{aligned}
\end{equation}

The Riesz rearrangement inequality in Theorem \ref{rri-cha}, implies that
\begin{equation}\label{rri-chi-step}
    \begin{aligned}
    \di \chi_{\{f\ge t\}}(x)\chi_{r^{1/q}K}(x-y)&\chi_{\{f\ge t\}}(y)dxdy\\
    &\le\di \chi_{\{f^{\star}\ge t\}}(x)\chi_{r^{1/q}K^{\star}}(x-y)\chi_{\{f^{\star}\ge t\}}(y)dxdy.
\end{aligned}
\end{equation}
Note that  \(\int_{\rn}\chi_{\{f\ge t\}}(x)dx<\infty\) for almost all $t>0$, as $f\in L^1(\rn)$. Combining \eqref{rri-chi-step} with the calculation in \eqref{cal-adi}, we have
\begin{align}\label{ri}
 \int_{\mathbb{R}^n}  \int_{\mathbb{R}^n} \min\{f(x),f(y)\}\|x-y\|_K^qdxdy \geq \int_{\mathbb{R}^n}  \int_{\mathbb{R}^n} \min\{f^\star(x),f^\star(y)\}\|x-y\|_{K^\star}^qdxdy.   
\end{align}

Moreover, if equality holds in \eqref{ri}, then there exists a null set $N\subset(0,\infty)^2$ such that equality holds in \eqref{rri-chi-step} for all $(t,r)\in(0,\infty)^2\setminus N$. For almost all $r\in (0,\infty)$, we have $(t,r)\in (0,\infty)^2\setminus N$ for almost all $t>0$. For $r$ sufficiently small, Theorem~\ref{rri-cha} guarantees the existence of a centered ellipsoid $D$ and $x\in\rn$  such that
\begin{align*}
\{f\ge t\}=x+a D \quad \text{and } \quad r^{1/q}K=b D.    
\end{align*}
Since $K=(\frac{|K|}{|D|})^{1/n}D$, the ellipsoid $D$ is independent of $r$ and $t$, which concludes the proof.
\end{proof}

Similarly, for $\alpha>n$, we show that $|\Laf|$ does not decrease under Schwarz symmetrals.

\begin{lem}\label{vii}
Let $\alpha > n$ and $f: \rn\to [0,\infty)$ be  integrable. If $|\La f|<\infty$, then
\begin{align*}
    |\La f|\geq |\La f^\star|.
\end{align*}
Equality holds if and only if the level set $\{f\ge t\}$ has measure zero or is homothetic to an ellipsoid for almost all $t>0$, up to null sets.
\end{lem}

\begin{proof}
By letting $q=\alpha-n$  in Lemma~\ref{q>n} and combining it with \eqref{a-i}, we obtain
\[\di\frac{\min\{f(x), f(y)\}}{\|x-y\|_K^{n-\alpha}}dxdy=n\tilde{V}_{\alpha}(K,\Laf)\ge n\tilde{V}_{\alpha}(K^{\star},\La f^{\star}).\]
With $K=\Laf$ and the dual mixed volume inequality \eqref{dmvi2}, we have
\begin{align*}
 |\La f|& =  \tilde{V}_\alpha (\La f, \La f) \\
 & \geq \tilde{V}_\alpha ((\La f)^\star, \La f^\star)\\
 & \geq |(\La f)^\star|^{\frac{n-\alpha}{n}}|\La f^\star|^\frac{\alpha}{n}\\
 &= |\La f|^{\frac{n-\alpha}{n}}|\La f^\star|^\frac{\alpha}{n},
\end{align*}
which implies that
\[|\Laf|\ge |\La f^{\star}|.\]
The equality case follows from Lemma \ref{q>n}.
\end{proof}

Building on the preceding lemmas, we now establish the affine chord Sobolev inequalities for $\alpha>n$.

\begin{proof}[Proof of Theorem 2]
 If the right-most side is $\infty$, the inequality holds trivially. We then assume that
\begin{align*}
\int_{\rn} \int_{\rn}   \frac{\min\{f(x),f(y)\}}{|x-y|^{n-\alpha}} dxdy <\infty.
\end{align*}
It follows from \eqref{a-i} and the dual mixed volume inequality \eqref{dmvi2} that
\begin{align*}
 \int_{\rn} \int_{\rn}   \frac{\min\{f(x),f(y)\}}{|x-y|^{n-\alpha}} dxdy = n \tilde{V}_{\alpha} (B^n, \La f) \geq n \omega^{\frac{n-\alpha}{n}} |\La f|^\frac{\alpha}{n}.   
\end{align*}
By the equality case in the dual mixed volume inequality,  $\La f$ must be a ball, which occurs when $f$ is radially symmetric.

For the first inequality, we may assume that $|\La f|$ is finite. Note that $\La f^{\star}$ is a ball. Then Lemma \ref{vii} and \eqref{a-i} imply that
\[n\omega_n^{\frac{n-\alpha}{n}}|\Laf|^{\frac{\alpha}{n}}\ge n\omega_n^{\frac{n-\alpha}{n}}|\La f^{\star}|^{\frac{\alpha}{n}}=n\tilde{V}_{\alpha}(B^n, \La f^{\star})=\di\frac{\min\{f^{\star}(x), f^{\star}(y)\}}{|x-y|^{n-\alpha}}dxdy.\]
Combining the chord Sobolev inequality \eqref{chord-Sob>n} with $\|f\|_1=\|f^{\star}\|_1$ and $\|f\|_{\infty}=\|f^{\star}\|_{\infty}$, we deduce
\[\sigma_{n,\alpha}\|f\|_1^{\frac{n+\alpha}{n}}\|f\|_{\infty}^{-\frac{\alpha}{n}}\leq \di\frac{\min\{f^{\star}(x), f^{\star}(y)\}}{|x-y|^{n-\alpha}}dxdy\leq n\omega_n^{\frac{n-\alpha}{n}}|\Laf|^{\frac{\alpha}{n}}.\]
The equality case follows from the optimizers of \eqref{chord-Sob>n} together with Lemma \ref{vii}.
\end{proof}

\section{Radial mean bodies for log-concave functions}\label{Radial mean bodies for log-concave functions}\label{Raf}

In this section, we begin by recalling some basic notation and properties of the classical radial mean bodies. Let $K\subset \mathbb{R}^n$ be a convex body. For $\alpha >-1$ and $u\in\sn$, Gardner and Zhang \cite{GZ} defined the radial $\alpha$-th mean body of $K$ by its radial function: 
\begin{align*}
    \rho_{\rR_\alpha K}(u)^\alpha =\frac{1}{|K|}\int_K\rho_{K-x}(u)^\alpha d x
\end{align*}
for $\alpha \neq 0$ and by
\begin{equation*}
    \log\rho_{\rR_0 K}(u)=\frac{1}{|K|}\int_K\log \rho_{K-x}(u) dx.
\end{equation*}

They further defined $\rR_{\infty} K$ consistently by
\[\rho_{\rR_{\infty} K}(u)=\lim_{\alpha\rightarrow \infty}\rho_{\Ra K}(u)=\max_{x\in K}\rho_{K-x}(u),\]
so that $\rR_{\infty} K$ coincides with the difference body of $K$, namely,
\[\rR_{\infty} K={\rm D}K=K-K.\]

\subsection{Affine isoperimetric inequalities for \texorpdfstring{$\boldsymbol{\Raf}$}{}}\hfill

Haddad and Ludwig \cite{HL1} proved the following affine isoperimetric inequality for $\Ra K$, with full equality characterization.

\begin{thm}\cite[Theorem 20]{HL1}\label{Theorem 20 HL}
    Let $K\subset\rn$ be a convex body. Then
    \begin{equation*}
        \begin{aligned}
            \frac{|\Ra K|}{|K|}&\leq \frac{|\Ra B^n|}{|B^n|},&& \alpha\in(-1,n),\\[0.5em]
            \frac{|\Ra K|}{|K|}&\geq \frac{|\Ra B^n|}{|B^n|},&& \alpha>n,
        \end{aligned}
    \end{equation*}
    with equality if and only if $K$ is an ellipsoid.
\end{thm}

For $\alpha>-1$, the authors \cite{BC} defined the radial mean body $\Raf$ for a non-zero, log-concave function $f\in L^1(\rn)$ by
\begin{equation}\label{rdm-superlevel set}
    \rho_{\Raf}(u)^{\alpha}=\int_0^{\|f\|_{\infty}}\rho_{\Ra \{f\ge t\} }(u)^\alpha d\mu_f(t),~~\alpha\neq 0,
\end{equation}
and
\begin{equation}\label{rdm0-superlevel set}
    \log\rho_{\rR_0f}(u)=\int_0^{\|f\|_{\infty}}\log\rho_{\rR_0 \{f\ge t\} }(u)d\mu_f(t),
\end{equation}
where $d\mu_f(t)=\frac{ |\{f\ge t\}| }{\|f\|_1}dt$ is a probability measure. See \cite{LSU} for the definition via Ball bodies.

The following result in \cite{BC} shows that $\Raf$ is proportional to the fractional polar projection body $\Pia f$ and the star-shaped set $\Laf$. Moreover, by the continuity of $\rho_{\Raf}(u)$ with respect to $\alpha$, the authors also gave an explicit formula for $\rho_{\rR_0f}$.

\begin{lem}\cite[Proposition 6.2, Lemma 6.4]{BC}
    For non-zero, log-concave $f\in L^1(\rn)$, 
    \begin{equation}\label{fppb}
        \rho_{\Raf}(u)^{\alpha}=\frac{-\alpha}{2\|f\|_1}\int_0^{\infty}r^{\alpha-1}\int_{\rn}|f(x)-f(x+ru)|dxdr
    \end{equation}
    for $\alpha\in(-1,0)$ and
    \begin{equation}\label{laf}
        \rho_{\Raf}(u)^{\alpha}=\frac{\alpha}{\|f\|_1}\int_0^{\infty}r^{\alpha-1}\int_{\rn}\min \{f(x), f(x+ru)\}dxdr
    \end{equation}
    for $\alpha>0$.

    For $\alpha=0$,
    \begin{equation*}
            \log\rho_{\rR_0f}(u)=-\gamma+\int_0^{\infty}\frac{1}{r}\Big(\frac{1}{\|f\|_1}\int_{\rn}\min\{f(x), f(x+ru)\}dx-e^{-r}\Big)dr,
    \end{equation*}
    where $\gamma=-\Gamma^{\prime}(1)$ is the Euler constant.
\end{lem}

The following lemma describes the affine covariance of $\Raf$, where ${\rm GL}(n)$ denotes the group of general linear transformations on $\rn$.

\begin{lem}\label{Raf-affco}
  Let $f\in L^1(\rn)$ be a non-zero, log-concave function. Then 
    \[\Ra(cf)=\Raf~~\text{and}~~\Ra (f\circ \lambda_y)=\Ra f,\]
    where $c>0$ is a constant, $y\in\rn$ and $f\circ \lambda_y(x)=f(x-y)$. Moreover, for $\phi\in {\rm GL}(n)$,
    \[\Ra(f\circ \phi^{-1})=\phi\Raf.\]
\end{lem}

\begin{proof}
    By \eqref{fppb} and \eqref{laf}, the radial mean body $\Raf$ is invariant under scaling and translations. It suffices to show $\Ra(f\circ \phi^{-1})=\phi\Raf$.

   We treat the case $\alpha>n$. The case $\alpha\in(-1,0)$ can be handled similarly.
 For $\phi\in {\rm GL}(n)$, note that $\|f\circ \phi^{-1}\|_1=|\det \phi|\|f\|_1$. Since
    \begin{equation*}
        \begin{aligned}
            \int_{\rn}\min\{f(\phi^{-1}x), f(\phi^{-1}(x+ru))\}dx=|\det \phi|\int_{\rn}\min\{f(x), f(x+r\phi^{-1}u)\}dx,
        \end{aligned}
    \end{equation*}
    by \eqref{laf}, we have
    \[\rho_{\Ra (f\circ\phi^{-1})}(u)^{\alpha}=\rho_{\Raf}(\phi^{-1}u)^{\alpha}=\rho_{\phi\Raf}(u)^{\alpha},\]
    which concludes the proof.
\end{proof}

Recall that the affine chord Sobolev inequalities can be written as
\begin{equation}\label{ASC--1n}
    \begin{aligned}
        2\sigma_{n,\alpha}\|f\|_{\frac{n}{n+\alpha}}&\leq n\omega_n^{\frac{n-\alpha}{n}}|\Pi^{*,-\alpha}f|^{\frac{\alpha}{n}},&&\alpha\in(-1,0),\\
        \sigma_{n,\alpha}\|f\|_{\frac{n}{n+\alpha}}&\ge n\omega_n^{\frac{n-\alpha}{n}}|\Laf|^{\frac{\alpha}{n}},&&\alpha\in(0,n)
    \end{aligned}
\end{equation}
and
\begin{equation}\label{ACS-n}
    \sigma_{n,\alpha}\|f\|_1^{\frac{n+\alpha}{n}}\|f\|_{\infty}^{-\frac{\alpha}{n}}\leq n\omega_n^{\frac{n-\alpha}{n}}|\Laf|^{\frac{\alpha}{n}},~~\alpha>n.
\end{equation}

By \eqref{laf} and the affine chord Sobolev inequalities, we establish the following affine isoperimetric inequalities for $\Raf$. We first focus on non-negative $\alpha$.

Note that, when $\alpha=n$ and $f\in L^{\frac{1}{2}}(\rn)$ is log-concave, the chord Sobolev inequalities become the following two trivial inequalities: 
\begin{equation}\label{chord Sob-n}
    \frac{\|f\|_1^2}{\|f\|_{\infty}}\leq \di \min\{f(x), f(y)\}dxdy\leq \|f\|_{\frac{1}{2}},
\end{equation}
where equality holds if and only if $f=c\chi_K$ for some convex body $K\subset\rn$ and constant $c>0$. Since
\[|\rR_n f|=\frac{1}{\|f\|_1}\di \min\{f(x), f(y)\}dxdy,\]
and $|\rR_n K|=|K|$ for any convex body $K\subset\rn$, both \eqref{-1n} and \eqref{n} hold for $\alpha=n$ and equality holds if and only if $f=c\chi_K$.

\begin{thm}
    Let $f\in L^{1}(\rn)$ be log-concave and non-zero. Then
    \begin{align}
        \frac{|\Raf|}{\|f\|_{\frac{1}{2}}\|f\|_1^{-1}}&\leq \frac{|\Ra B^n|}{|B^n|},\quad\alpha\in(-1,0)\cup(0,n), \label{-1n}\\
        \frac{|\Ra f|}{\|f\|_1\|f\|_{\infty}^{-1}}&\geq \frac{|\Ra B^n|}{|B^n|},\quad\alpha>n\label{n}.
    \end{align}
    In both cases, equality holds if and only if $f$ is a constant multiple of the characteristic function of an ellipsoid.
\end{thm}

\begin{proof}
    Inequality \eqref{n} and the equality characterization follow from \eqref{ACS-n} directly. We turn to $\alpha\in(-1,0)\cup(0,n)$. By the affine chord Sobolev inequality \eqref{ASC--1n}, we obtain
    \[\frac{|\Raf|}{\|f\|_{\frac{n}{n+\alpha}}^{n/\alpha}\|f\|_1^{-n/\alpha}}\leq \frac{|\Ra B^n|}{|B^n|}.\]

    For brevity, we denote by $\bar f(x)=f(x)/\|f\|_1$. Then $\bar f(x)dx$ is a probability measure on $\rn$. Combining this with the monotonicity of $L^p$ norms, for $\alpha\in(-1,n)\setminus\{0\}$, we have
    \begin{equation}\label{L^p mono}
        \begin{aligned}
            \|f\|_{\frac{n}{n+\alpha}}^{n/\alpha}\|f\|_1^{-n/\alpha}&=\Big(\int_{\rn}\bar f(x)^{\frac{n}{n+\alpha}} dx\Big)^{\frac{n+\alpha}{\alpha}}\\
            &=\Big(\int_{\rn}\big(\bar f(x)^{-1}\big)^{\frac{\alpha}{n+\alpha}} \bar f(x)dx\Big)^{\frac{n+\alpha}{\alpha}}\leq \Big(\int_{\rn}\bar f(x)^{\frac{1}{2}}dx\Big)^2=\|f\|_{\frac{1}{2}}\|f\|_1^{-1},
        \end{aligned}
    \end{equation}
    where equality holds if and only if $f=c\chi_K$ for some $c>0$ and convex body $K\subset\rn$. Therefore, we obtain
    \begin{equation}\label{step}
        \frac{|\Raf|}{\|f\|_{\frac{1}{2}}\|f\|_1^{-1}}\leq \frac{|\Ra B^n|}{|B^n|}.
    \end{equation}
    If equality holds, then \eqref{L^p mono} implies that $f=c\chi_K$ where $K\subset\rn$ is a convex body and $c>0$. Hence \eqref{step} reduces to the geometric inequality \eqref{Theorem 20 HL}, whose equality cases imply that $K$ is an ellipsoid.
\end{proof}

Next we consider the case $\alpha=0$. For non-zero, log-concave $f\in L^{\frac12}(\rn)$ and fixed $u\in\sn$, the radial function $\rho_{\Raf}(u)$ is continuous in $\alpha>-1$, as shown in \cite{BC}. Therefore, taking limits on both sides of \eqref{step} yields the affine isoperimetric inequality for $\rR_0 f$, but the equality characterization does not follow from this argument. In the following theorem, we use \eqref{Theorem 20 HL} together with the definition of $\rR_0 f$ to establish the inequality with its equality characterization.

\begin{thm}
    Let $f\in L^{\frac{1}{2}}(\rn)$ be a non-zero, log-concave function. Then
    \[\frac{|\rR_0f|}{\|f\|_{\frac{1}{2}}\|f\|_1^{-1}}\leq \frac{|\rR_0 B^n|}{|B^n|},\]
    where equality holds if and only if $f$ is a constant multiple of the characteristic function of an ellipsoid.
\end{thm}

\begin{proof}
    By \eqref{rdm0-superlevel set},
    \[\rho_{\rR_0f}(u)^n=\exp\Big(\int_0^{\|f\|_{\infty}}\log \rho_{\rR_0 \{f\ge t\} }(u)^nd\mu_f(t)\Big).\]
    Since $e^{t}$ is a convex function on $\tr$ and $\mu_f$ is a probability on $\tr$, it follows from Jensen's inequality that
    \begin{equation*}
        \rho_{\rR_0f}(u)^n\leq \int_0^{\|f\|_{\infty}}\rho_{\rR_0\{f\ge t\}}(u)^nd\mu_f(t).
    \end{equation*}
    This, together with Theorem \ref{Theorem 20 HL}, implies that
    \begin{equation*}
        \begin{aligned}
            |\rR_0f|=\frac{1}{n}\int_{\sn}\rho_{\rR_0 f}(u)^ndu&\leq \frac{1}{n}\int_{\sn}\int_0^{\|f\|_{\infty}}\rho_{\rR_0 \{f\ge t\} }(u)^n d\mu_f(t)du\\
            &=\int_0^{\|f\|_{\infty}}|\rR_0\{f \ge t\}|d\mu_f(t)\\&\leq \frac{|\rR_0 B^n|}{|B^n|}\int_0^{\|f\|_{\infty}}\frac{|\{f\ge t\}|^2}{\|f\|_1}dt\\
            &=\frac{|\rR_0 B^n|}{|B^n|}\di \frac{\min\{f(x), f(y)\} }{\|f\|_1}dxdy.
        \end{aligned}
    \end{equation*}
    By \eqref{chord Sob-n}, we obtain
    \begin{equation}\label{0-step}
        \frac{|\rR_0f|}{\|f\|_{\frac{1}{2}}\|f\|_1^{-1}}\leq \frac{|\rR_0 B^n|}{|B^n|}.
    \end{equation}
    If equality holds, then
    \[ \di \min\{f(x), f(y)\}dxdy=\|f\|_{\frac{1}{2}},\]
    which implies that $f=c\chi_K$ for some convex body $K\subset\rn$ and $c>0$. Hence \eqref{0-step} reduces to the geometric inequality in \eqref{Theorem 20 HL}. By the equality characterization of \eqref{Theorem 20 HL}, $K$ is an ellipsoid.
\end{proof}

\subsection{Monotonicity of \texorpdfstring{$\boldsymbol{\Raf}$}{}}\hfill

To consider the monotonicity of radial mean bodies, we recall the following result by Haddad and Ludwig \cite{HL1}.

\begin{lem}\cite[Theorem 3]{HL1}\label{R-1}
    For $\alpha\in(-1,0)$ and $f\in L^1(\rn)$ with bounded variation,
    \[\lim_{\alpha\rightarrow -1^+}(1+\alpha)\|f\|_1\rho_{\Raf}(u)^{\alpha}=\rho_{\Pi^* f}(u)^{-1}\]
    for $u\in\sn$.
\end{lem}

We mainly use Lemma \ref{R-1} for log-concave functions $f\in L^1(\rn)$. By \cite[Lemma 5.4]{BC}, $\tH^{n-1}(\partial \{f\ge t\})$ is an integrable function on $\tr$. Hence the coarea formula \cite[Theorem 5.9]{EG}, implies that $f$ has bounded variation. The opposite endpoint $\alpha=\infty$ will be considered below; see also \cite{LSU} for a different approach.

{\begin{lem}\label{R_infty}
    Let $f\in L^1(\rn)$ be log-concave and non-zero. If $f$ has compact support, then
    \[\lim_{\alpha\rightarrow\infty}\rho_{\Ra f}(u)=\sup_{t>0}\rho_{{\rm D}\{f\ge t\}}(u)=\rho_{{\rm D}\, \supp (f)}(u),\]
    for $u\in\sn$.
\end{lem}}

\begin{proof} 
    Since Jensen's inequality implies that $\rR_{\alpha} K\subset {\rm D} K$ for $\alpha>-1$, we have
    \[\rho_{\Ra f}(u)\leq \Big(\int_{0}^{\infty}\rho_{{\rm D}\{f\ge t\}}(u)^{\alpha}d\mu_f(t)\Big)^{1/\alpha}.\]
    Letting $\alpha\rightarrow\infty$, we obtain
    \begin{equation}\label{max 1}
        \limsup_{\alpha\rightarrow\infty}\rho_{\Ra f}(u)\leq\sup_{t>0}\rho_{{\rm D}\{f\ge t\}}(u)=\rho_{{\rm D}\, \supp (f)}(u).
    \end{equation}
    On the other hand, by the reversed inclusion ${\rm D}K\subset c_{n,\alpha}\rR_{\alpha} K$ in \eqref{GZ-mian5.5},
    \[\rho_{\Ra f}(u)\ge c_{n,\alpha}^{-1}\Big(\int_{0}^{\infty}\rho_{{\rm D}\{f\ge t\}}(u)^{\alpha}d\mu_f(t)\Big)^{1/\alpha},\]
    where $c_{n,\alpha}=(nB(\alpha+1,n))^{-1/\alpha}$. Note that $\lim_{\alpha\rightarrow\infty}c_{n,\alpha}=1$. Therefore we have
    \begin{equation}\label{max 2}
        \liminf_{\alpha\rightarrow\infty}\rho_{\rR_{\alpha} f}(u)\ge \sup_{t>0}\rho_{{\rm D}\{f\ge t\}}(u)=\rho_{{\rm D}\,\supp(f)}(u).
    \end{equation}
    Combining \eqref{max 1} with \eqref{max 2}, we then conclude the proof.
\end{proof}

This yields the natural definition of $\rR_{\infty}f$ given below.

\begin{defn}
    Let $f\in L^1(\rn)$ be log-concave with compact support. The body $\rR_{\infty} f$ is given by
    \[\rR_{\infty} f={\rm D}\,\supp(f).\]
    If $f$ is not compactly supported, $\rR_{\infty} f=\rn$.
\end{defn}

Recall that in the geometric case, Jensen's inequality implies that

\begin{equation}\label{inclusion1}
    \Ra K\subset \rR_{\beta} K\subset {\rm D}K,
\end{equation}
where Gardner and Zhang \cite{GZ} further showed that
\begin{equation}\label{inclusion2}
    \Pi^*K\subset ((1+\alpha)|K|)^{1/\alpha}\Ra K\subset ((1+\beta)|K|)^{1/\beta}\rR_{\beta} K\subset {\rm D}K
\end{equation}
for $-1<\alpha<\beta$.

Applying Jensen’s inequality, together with \eqref{inclusion1}, we immediately obtain
\[\rR_{\alpha} f\subset \rR_{\beta} f\subset \rR_{\infty} f,\]
for $-1<\alpha<\beta<\infty$. The following theorem provides a functional analogue of \eqref{inclusion2}.

\begin{thm}
    For non-zero, log-concave $f\in L^1(\rn)$,
    \[\|f\|_{\infty}\Pi^{*} f\subset \Big((1+\alpha)\frac{\|f\|_1}{\|f\|_{\infty}}\Big)^{1/\alpha} \Raf\subset \Big((1+\beta)\frac{\|f\|_1}{\|f\|_{\infty}}\Big)^{1/\beta}\rR_{\beta} f\subset \rR_{\infty} f,\]
    for $-1<\alpha<\beta<\infty$.
\end{thm}

\begin{proof}
    If $f$ is not compactly supported, the right-most inclusion is trivial since $\rR_{\infty} f=\rn$. We now assume that $f$ has compact support.
    
    By the definition of $\Raf$ in \eqref{rdm-superlevel set},
    \begin{equation*}
        \begin{aligned}
            \Big((1+\alpha)\frac{\|f\|_1}{\|f\|_{\infty}}\Big)^{1/\alpha}\rho_{\Raf}(u)&=\Big(\frac{1}{\|f\|_{\infty}}\int_0^{\|f\|_{\infty}}(1+\alpha)|\{f\ge t\}|\rho_{\Ra\{f\ge t\}}(u)^{\alpha}dt\Big)^{1/\alpha}\\
            &=\Big(\frac{1}{\|f\|_{\infty}}\int_0^{\|f\|_{\infty}}\rho_{l_{\alpha,t}\Ra \{f\ge t\}}(u)^{\alpha}dt\Big)^{1/\alpha},
        \end{aligned}
    \end{equation*}
    where $l_{\alpha,t}=((1+\alpha)|\{f\ge t\}|)^{1/\alpha}$.
    It follows from \eqref{inclusion2} that
    \[l_{\alpha,t}\rho_{\Ra \{f\ge t\}}(u)\leq l_{\beta,t} \rho_{\rR_{\beta} \{f\ge t\}}(u),\]
    for $-1<\alpha<\beta$.

    Jensen's inequality then implies that
    \begin{equation*}
        \begin{aligned}
            \Big((1+\alpha)\frac{\|f\|_1}{\|f\|_{\infty}}\Big)^{1/\alpha}\rho_{\Raf}(u)&\leq \Big(\frac{1}{\|f\|_{\infty}}\int_0^{\|f\|_{\infty}}\rho_{l_{\beta,t}\rR_{\beta} \{f\ge t\}}(u)^{\alpha}dt\Big)^{1/\alpha}\\
            &\leq \Big(\frac{1}{\|f\|_{\infty}}\int_0^{\|f\|_{\infty}}\rho_{l_{\beta,t}\rR_{\beta} \{f\ge t\}}(u)^{\beta}dt\Big)^{1/\beta}\\
            &=\Big((1+\beta)\frac{\|f\|_1}{\|f\|_{\infty}}\Big)^{1/\beta}\rho_{\rR_{\beta} f}(u),
        \end{aligned}
    \end{equation*}
    which gives
    \[\Big((1+\alpha)\frac{\|f\|_1}{\|f\|_{\infty}}\Big)^{1/\alpha} \Raf\subset \Big((1+\beta)\frac{\|f\|_1}{\|f\|_{\infty}}\Big)^{1/\beta}\rR_{\beta} f.\]

    Letting $\beta\rightarrow\infty$, together with $\lim_{\beta\rightarrow\infty}(1+\beta)^{1/\beta}=1$, we obtain the right-most inclusion. Letting $\alpha\rightarrow -1^+$, the left-most inclusion follows from Lemma \ref{R-1}.
\end{proof}

\section{Reverse inclusion for radial mean bodies}\label{Reversed inequalities}


In the previous section, we recalled the monotonicity relations \eqref{inclusion1} and \eqref{inclusion2} for radial mean bodies and extended them to their functional analogues. In this section, we turn to the reverse inclusion, which is considerably more delicate. We begin by stating the following theorem established by Gardner and Zhang \cite{GZ}.

\begin{thm}\cite[Theorem 5.5]{GZ}\label{GZ-reverse}
    Let $K$ be a convex body in $\rn$. If $-1<\alpha<\beta$, then
    \[{\rm D}K\subset c_{n,\beta}\rR_{\beta} K\subset c_{n,\alpha} \Ra K\subset n|K|\Pi^* K.\]
    In each inclusion equality holds if and only if $K$ is a simplex.
\end{thm}

When $\beta=n$, the left-most inclusion implies the Rogers-Shephard inequality:
\[|{\rm D} K|\leq \binom{2n}{n} |K|.\]
When $\alpha=n$, the right-most inclusion yields Zhang's inequality:
\[n^{-n}\binom{2n}{n}=|\triangle|^{n-1}|\Pi^*\triangle|\leq |K|^{n-1}|\Pi^* K|,\]
where $\triangle$ denotes a simplex in $\rn$.

To extend Theorem \ref{GZ-reverse} to the functional setting, we rely on the following result from \cite{HL3}.
It generalizes the classical result of Marshall, Olkin and Proschan \cite{MOP}, whose proof was later simplified by Milman and Pajor \cite{MP}. 

{\begin{lem}\cite[Lemma 17]{HL3}\label{lemma_aux}
Let $\omega: [0,\infty)\to [0,\infty)$    be decreasing with
\begin{align*}
    0<\int_0^\infty r^{\alpha -1}\omega(r) d r<\infty
\end{align*}
for $\alpha >0$ and
\[0<\int_0^{\infty}r^{\alpha-1}(\omega(0)-\omega(r))dr<\infty\]
for $\alpha\in(-1,0)$. If $\varphi: [0,\infty)\to[0,\infty)$ is non-zero, with $\varphi(0)=0$, and such that $r\mapsto \varphi(r)$ and $r\mapsto {\varphi(r)}/{r}$ are increasing on $(0,\infty)$, then
        \begin{equation*}
            \zeta(\alpha)=\left\{
                \begin{aligned}
                &\left(\frac{\int_0^{\infty}r^{\alpha-1}\omega(\varphi(r))dr}{\int_0^{\infty}r^{\alpha-1}\omega(r)dr}\right)^{1/\alpha}&&\text{for}~~\alpha>0,\\
                &\exp\left(\int_0^{\infty}\frac{\omega(\varphi(r))-\omega(r)}{r\omega(0)}dr\right)&&\text{for}~~\alpha=0,\\
                &\left(\frac{\int_0^{\infty}r^{\alpha-1}(\omega(\varphi(r))-\omega(0))dr}{\int_0^{\alpha-1}(\omega(r)-\omega(0))dr}\right)^{1/\alpha}&&\text{for}~~-1<\alpha<0,
                \end{aligned}\right.
        \end{equation*}
is a continuous, decreasing function of $\alpha$ on $(-1, \infty)$. Moreover, $\zeta$ is constant on $(-1, \infty)$ if $\varphi(r)=\lambda r$ on $[0, \infty)$ for some $\lambda >0$.
\end{lem}}

In particular, taking $\omega(r)=e^{-r}$ and $\omega(r)=(1-sr)_+^{1/s}$ with $s>0$ yields the classical analytic continuation formulas for the Gamma and Beta functions,
\begin{equation*}
    \Gamma(\alpha)=\left\{
                \begin{aligned}
                &\int_0^{\infty}r^{\alpha-1}e^{-r}dr&&\text{for}~~\alpha>0,\\
                &\int_0^{\infty}r^{\alpha-1}(e^{-r}-1)dr&&\text{for}~~-1<\alpha<0,\\
                \end{aligned}\right.
\end{equation*}
and
\begin{equation}\label{beta fcn}
            s^{-\alpha}B(\alpha,1+\frac{1}{s})=\left\{
                \begin{aligned}
                &\int_0^{\infty}r^{\alpha-1}(1-sr)_+^{1/s}dr&&\text{for}~~\alpha>0,\\
                &\int_0^{\infty}r^{\alpha-1}((1-sr)_+^{1/s}-1)dr&&\text{for}~~-1<\alpha<0.\\
                \end{aligned}\right.
\end{equation}

Under the assumptions of Lemma \ref{lemma_aux} with strictly decreasing $\omega(r)$, if $\zeta(\alpha)$ is constant on $(-1,\infty)$, say $\zeta(\alpha)=\lambda$, then $\varphi(r)=\lambda r$ for all $r\ge 0$. The result follows directly from Milman and Pajor \cite{MP}, but we provide a proof for completeness.

\begin{lem}\label{lemma_aux_iff}
Let $\omega$, $\varphi$, and $\zeta$ be as in Lemma \ref{lemma_aux}. Suppose $\omega$ is strictly decreasing and $-1<\alpha<\beta<\infty$. Then $\zeta(\alpha)=\zeta(\beta)$  implies $\varphi(r)=\vartheta r$ with $\vartheta=\zeta(\alpha)^{-1}$.
\end{lem}

\begin{proof}
    Let $\vartheta =\zeta(\alpha)^{-1}$. We first consider $0<\alpha<\beta<\infty$. For $\alpha>0$, we then have
    \[\int_0^{\infty}\omega(\vartheta r)r^{\alpha-1}dr=\int_0^{\infty}\omega(\varphi(r))r^{\alpha-1}dr.\]
    We set $\Psi(t)=\int_t^{\infty}r^{\alpha-1}(\omega(\vartheta r)-\omega(\varphi(r)))dr$, where $\Psi(0)=\Psi(\infty)=0$. Since $\varphi(r)/r$ is increasing while $\omega$ is decreasing, we have
    \[\omega(\varphi(r))> \omega(\vartheta r),~~\text{if}~~{\varphi(r)}/{r}< \vartheta,\]
    and
    \[\omega(\varphi(r))\leq \omega(\vartheta r),~~\text{if}~~{\varphi(r)}/{r}\ge \vartheta.\]
    This implies the sign of $\omega(\vartheta t)-\omega(\varphi(t))=-\Psi^{\prime}(t)$. Therefore, $\Psi$ is first increasing and then decreasing, which yields that $\Psi(t)\ge 0$ for all $t\ge 0$. That is
    \[\int_t^{\infty}\omega(\vartheta r)r^{\alpha-1}dr\ge \int_t^{\infty}\omega(\varphi(r))r^{\alpha-1}dr.\]

    For $\beta>\alpha>0$, we have
    \begin{equation*}
        \begin{aligned}
            \int_0^{\infty}\omega(\varphi(r))r^{\beta-1}dr&=(\beta-\alpha)\int_0^{\infty}t^{\beta-\alpha-1}\Big(\int_t^{\infty}\omega(\varphi(r))r^{\alpha-1}dr\Big)dt\\
            &\leq (\beta-\alpha)\int_0^{\infty}t^{\beta-\alpha-1}\Big(\int_t^{\infty}\omega(\vartheta  r)r^{\alpha-1}dr\Big)dt\\
            &=\int_0^{\infty}\omega(\vartheta r)r^{\beta-1}dr=\frac{1}{\vartheta^{\beta}}\int_0^{\infty}\omega(r)r^{\beta-1}dr,
        \end{aligned}
    \end{equation*}
    which implies $\zeta(\alpha)\ge \zeta(\beta)$. If $\zeta(\alpha)=\zeta(\beta)=\vartheta^{-1}$, equality holds in the inequality above, which yields that $\Psi(t)=0$ for all $t\ge 0$. Then we have
    \[\Psi^{\prime}(t)=t^{\alpha-1}(\omega(\varphi(t))-\omega(\vartheta t))=0.\]
    Since $\omega$ is strictly decreasing, we obtain $\varphi(t)=\vartheta t$.

    For $-1<\alpha<\beta<0$, the proof is the same. Then we consider $-1<\alpha<0<\beta<\infty$. Since $\zeta$ is decreasing, we have
    \[\zeta(\beta)\leq \zeta(0)\leq \zeta(\alpha).\]
    As assumed $\zeta(\alpha)=\zeta(\beta)=\lambda^{-1}$, we obtain that $\zeta$ is a constant on $(0,\beta)$, by the first case, we have $\varphi(t)=\vartheta t$.
\end{proof}

\subsection{Results for log-concave functions}\label{Reversed inequalities-log}\hspace{\fill}

In this subsection, we focus on the case of log-concave functions. Note that there are two natural ways to embed a convex body into log-concave function class:
\[K\mapsto \chi_K,\qquad K\mapsto e^{-\|\cdot\|_K}.\]
Using the second embedding, Langharst, Mar\'in Sola and Ulivelli \cite{LSU} proved a reversed inclusion for $\Raf$. For completeness, we give a proof in the present framework. The following result from \cite{ABG} serves as a main tool in the equality characterization.

\begin{lem}\cite[Lemma 4.2]{ABG}\label{iff-abg}
    Let $f\in L^1(\rn)$ be non-zero and log-concave such that $f(o)=\|f\|_{\infty}$. Let
    \[g(x)=\int_{\rn}\min\{f(y), f(y+x)\}dy.\]
    Then for every $x\in\rn$ and $\lambda \in (0,1)$, $g(\lambda x)=g(o)^{1-\lambda}g(x)^{\lambda}$ if and only if $f(x)=\|f\|_{\infty}e^{-\|x\|_{\triangle}}$ with $\triangle$ a simplex containing the origin.
\end{lem}

\begin{thm}\label{t4}
    Let $f\in L^1(\rn)$ be a non-zero, log-concave function. Then
    \begin{align}\label{inc}
     \frac{1}{\Gamma(\beta+1)^{1/\beta}}\rR_{\beta} f\subset \frac{1}{\Gamma(\alpha+1)^{1/\alpha}}\Raf\subset 2\|f\|_1 \Pi^* f   
    \end{align}
    for $-1<\alpha<\beta<\infty$. In each inclusion, there is equality { if and only if} {$f(x)=a e^{-\|x-x_0\|_{\triangle}}$ where $a=\|f\|_{\infty}$ and $\triangle$ is a simplex containing $x_0$.}
\end{thm}

\begin{proof}
    By Lemma \ref{abg2.1},
    \[g(ru)=\int_{\rn}\min\{f(x), f(x+ru)\}dx\]
    is an even, log-concave function and attains maximum at $r=0$.

    For fixed $u\in\sn$, we apply Lemma \ref{lemma_aux} to $\omega(r)=g(o)e^{-r}$ and $\varphi(r)=-\log(g(ru)/g(o))$. We observe that
    \[\zeta(\alpha)=\bigg(\frac{\int_0^{\infty}r^{\alpha-1}g(ru)dr}{\int_0^{\infty}r^{\alpha-1}g(o)e^{-r}dr}\bigg)^{1/\alpha}=\frac{\rho_{\Raf}(u)}{\Gamma(\alpha+1)^{1/\alpha}}\]
    for $\alpha>0$.

    For $\alpha\in(-1,0)$, note that
    \begin{equation}\label{cal1}
        \begin{aligned}
            g(o)-g(ru)&=\|f\|_1-\int_{\rn}\min\{f(x),f(x+ru)\}dx\\
                      &=\frac{1}{2}\int_{\rn}f(x)dx+\frac{1}{2}\int_{\rn}f(x+ru)dx-\int_{\rn}\min\{f(x), f(x+ru)\}dx\\
                      &=\frac{1}{2}\int_{\rn}(f(x)-f(x+ru))_+dx+\frac{1}{2}\int_{\rn}(f(x+ru)-f(x))_+dx\\
                      &=\frac{1}{2}\int_{\rn}|f(x)-f(x+ru)|dx,
        \end{aligned}
    \end{equation}
    where $h(x)_+$ denotes the positive part of the function $h(x)$. Then we have
   \[\zeta(\alpha)=\bigg(\frac{\int_0^{\infty}r^{\alpha-1}(g(o)-g(ru))dr}{\int_0^{\infty}r^{\alpha-1}g(o)(1-e^{-r})dr}\bigg)^{1/\alpha}=\frac{\rho_{\Raf}(u)}{\Gamma(\alpha+1)^{1/\alpha}}\]
   for $\alpha\in(-1,0)$.

  Therefore Lemma \ref{lemma_aux} implies that
  \begin{equation}\label{inclusion-step}
      \frac{1}{\Gamma(\beta+1)^{1/\beta}}\rR_{\beta} f\subset\frac{1}{\Gamma(\alpha+1)^{1/\alpha}}\Raf.
  \end{equation}
  Letting $\alpha\rightarrow -1$, Lemma \ref{R-1} gives the right inclusion. {Since $\omega(r)=e^{-r}$ is strictly decreasing, by Lemma \ref{lemma_aux_iff}, equality holds in \eqref{inclusion-step} if and only if} $g(ru)=g(o)e^{-\lambda r}$ for some $\lambda >0$. Then Lemma \ref{iff-abg} shows that $f(x)=ae^{-\|x-x_0\|_{\triangle}}$ for $a=\|f\|_{\infty}$ and $x_0\in\triangle$, where $\triangle$ is a simplex containing the origin.
\end{proof}

Note that, in \eqref{inclusion-step}, when letting $\beta\rightarrow\infty$, the left-hand side  degenerates to  $\{o\}$. This observation motivates the consideration of the embedding $K\mapsto \chi_K$ in studying functional extensions of Theorem~\ref{GZ-reverse}. The following counterexample shows, however, that this extension fails for general log-concave functions, which naturally leads to the study of $s$-concave functions in the next subsection.

\begin{ex}\label{ex6}
    Let $n=1$ and $f(x)=2x\chi_{[0,1]}(x)$. It is clear that $\log f(x)=\log x+\log 2$ for $x\in[0,1]$, which is concave on $\supp(f)$. Then we have
    \[g(r)=(1-|r|)_+^2.\]
    In this case, we consider $\beta=1$ and $\alpha=\frac{1}{2}$ and obtain that 
\[c_{1,1}\rR_1 f=\Big[-\frac{2}{3},\frac{2}{3}\Big], \quad c_{1,\frac{1}{2}}\rR_{\frac{1}{2}} f=\Big[-\frac{16}{25}, \frac{16}{{25}}\Big],\] 
which is not consistent with the desired result $c_{n,\beta}\rR_{\beta} f\subset c_{n,\alpha}\rR_{\alpha} f$.
\end{ex}

\subsection{Results for \texorpdfstring{$s$}{}-concave functions}\hspace{\fill}

In this subsection, we focus on the case of $s$-concave functions with $s>0$. We continue to use the notation
\[g(ru)=\int_{\rn}\min\{f(x), f(x+ru)\}dx.\]
The following lemma is a direct corollary from the Borell-Brascamp-Lieb inequality, which is also mentioned in \cite[Theorem 11.2]{Gar1}.

{ \begin{lem}\label{1/ns-concave}
    Let $s>0$ and $u\in\sn$ be fixed. For an $s$-concave, integrable function $f$, the function $r\mapsto g(ru)$ is $s/(ns+1)$-concave on $\mathbb{R}$.
\end{lem}}

\begin{proof}
    We denote by
    \[N_f(x,r)=\min\{f(x), f(x+ru)\}.\]
    
    First, we prove that $N_f(x,r)$ is $s$-concave on $\rn\times\tr$. For $x_0,x_1\in\rn$, $r_0, r_1\in\tr$ and $\lambda\in[0,1]$, we denote $x_{\lambda}=(1-\lambda)x_0+\lambda x_1$ and $r_{\lambda}=(1-\lambda)r_0+\lambda r_1$. Note that
    \[N_f(x_{\lambda}, r_{\lambda})^s=\min\{f(x_{\lambda})^s, f(x_{\lambda}+r_{\lambda}u)^s\}.\]

    Since $f$ is $s$-concave and $\min\{\cdot, \cdot\}$ is a concave function on $\tr\times\tr$, we have
    \begin{equation*}
        \begin{aligned}
            N_f(x_{\lambda},r_{\lambda})^s&\ge \min\{(1-\lambda) f(x_0)^s+\lambda f(x_1)^s, (1-\lambda)f(x_0+r_0u)^s+\lambda f(x_1+r_1 u)^s\}\\
            &\ge (1-\lambda)\min\{f(x_0)^s, f(x_0+r_0 u)^s\}+\lambda \min\{f(x_1)^s, f(x_1+r_1 u)^s\}\\
            &=(1-\lambda)N_f(x_0,r_0)^s+\lambda N_f(x_1, r_1)^s,
        \end{aligned}
    \end{equation*}
    which implies the $s$-concavity of $N_f$.

    For $\lambda\in[0,1]$, and fixed $r_0,r_1\in\tr$, we set
    \[F(x)=N_f(x,r_0),~~G(x)=N_f(x,r_1),~~H(x)=N_f(x,r_{\lambda}),\]
    where $r_{\lambda}=(1-\lambda)r_0+\lambda r_1$. By the $s$-concavity of $N_f$, we have
    \[H((1-\lambda)x+\lambda y)\ge \mathcal{M}_{\lambda}^s(F(x), G(y)),\]
    where $\mathcal{M}_\lambda^s$ denotes the $s$-mean (see Section \ref{log s}). The Borell-Brascamp-Lieb inequality \eqref{BBL} then implies that
    \begin{equation}\label{BBL-application}
        \int_{\rn}H\ge \mathcal{M}_{\lambda}^{\frac{s}{ns+1}}\Big(\int_{\rn}F, \int_{\rn}G\Big),
    \end{equation}
    that is, the $\frac{s}{ns+1}$-concavity of $g(ru)$ in $r$.
\end{proof}

We begin the proof of Theorem 3 by stating the inclusion result as the following theorem. The equality case will be addressed separately. We note that the inclusion part was also proved in \cite{LSU} by the Ball-body framework. We include a direct proof below, which is also needed for the equality characterization.

\begin{thm}\label{main5-section}
    Let $s>0$ and $f\in L^1(\rn)$ be a non-zero $s$-concave function with compact support. If $-1<\alpha< \beta$, then
    \begin{equation}\label{s-reverse in steps}
        \rR_{\infty} f\subset c_{n,\beta}(s)\rR_{\beta} f\subset c_{n,\alpha}(s)\Raf\subset (n+\frac{1}{s})\|f\|_1 \Pi^* f,
    \end{equation}
    where $c_{n,\alpha}(s)=((n+\frac{1}{s})B(\alpha+1, n+\frac{1}{s}))^{-1/\alpha}$.
\end{thm}

\begin{proof}
     By Lemma \ref{1/ns-concave}, the function $g(ru)=\int_{\rn}\min\{f(x), f(x+ru)\}dx$ is $s^{\prime}$-concave with $s^{\prime}=s/(ns+1)$. We apply Lemma \ref{lemma_aux} with $\omega(r)=g(0)(1-s^{\prime}r)_+^{1/s^{\prime}}$ and $\varphi(r)=\frac{1}{s^{\prime}}(1-(g(ru)/g(0))^{s^{\prime}})$.
    
    Therefore, we obtain that
    \[\zeta(\alpha)=\bigg(\frac{\int_0^{\infty}r^{\alpha-1}g(ru)dr}{\int_0^{\infty }g(0) r^{\alpha-1}(1-s^{\prime}r)_+^{1/s^{\prime}}dr}\bigg)^{1/\alpha}=\frac{s^{\prime}\rho_{\Raf}(u)}{(\alpha B(\alpha, 1+\frac{1}{s^{\prime}}))^{1/\alpha}}\]
    for $\alpha>0$ and
    \[\zeta(\alpha)=\bigg(\frac{\int_0^{\infty}r^{\alpha-1}(g(0)-g(ru))dr}{\int_0^{\infty }g(0) r^{\alpha-1}((1-s^{\prime}r)_+^{1/s^{\prime}}-1)dr}\bigg)^{1/\alpha}=\frac{s^{\prime}\rho_{\Raf}(u)}{(\alpha B(\alpha, 1+\frac{1}{s^{\prime}}))^{1/\alpha}}\]
    for $-1<\alpha<0$, where we use \eqref{beta fcn} and \eqref{cal1}. Note that 
    $$\Big(\alpha B(\alpha, 1+\frac{1}{s^{\prime}})\Big)^{-1/\alpha}=\Big(\frac{1}{s^{\prime}}B(\alpha+1,\frac{1}{s^{\prime}})\Big)^{-1/\alpha}=c_{n,\alpha}(s).$$ 
    Lemma \ref{lemma_aux} then implies that
    \begin{equation}\label{reversed-inclusion-fcn2 step}
        c_{n,\beta}(s)\rR_{\beta} f\subset c_{n,\alpha}(s)\Raf,
    \end{equation}
    for $-1<\alpha<\beta$. 

     Since Lemma \ref{R_infty} ensures that $\rR_{\beta} f \to \rR_{\infty} f$ and $c_{n,\beta}(s) \to 1$ as $\beta\rightarrow \infty$, the left-most inclusion follows directly from \eqref{reversed-inclusion-fcn2 step} by letting $\beta\rightarrow\infty$.
    
    By Lemma \ref{R-1} and $c_{n,\alpha}(s)\sim\frac{\alpha+1}{n+1/s}$ when $\alpha\rightarrow -1$, we have
    \[\lim_{\alpha\rightarrow -1^+}c_{n,\alpha}(s)\rho_{\Raf}(u) =(n+\frac{1}{s})\|f\|_1\rho_{\Pi^*f}(u),\]
    which yields the right-most inclusion.
\end{proof}

The following theorem demonstrates the sharpness of \eqref{s-reverse in steps}.

\begin{thm}
    Let $s>0$. If $f(x)=a(1-\|x-x_0\|_{\triangle-x_0})_+^{1/s}$, where $a=\|f\|_{\infty}>0$ and $\triangle\subset\rn$ is a simplex containing $x_0$, then 
    \[\rR_{\infty} f= c_{n,\beta}(s)\rR_{\beta} f= c_{n,\alpha}(s)\Raf= (n+\frac{1}{s})\|f\|_1 \Pi^* f,\]
    where $-1<\alpha<\beta<\infty$ and $c_{n,\alpha}(s)$ is given in Theorem \ref{main5-section}.
\end{thm}

\begin{proof}
    By Lemma \ref{Raf-affco}, $\Raf$ is invariant under scaling and translation. We may assume,  without loss of generality, that $a=1$. Then
    \begin{equation*}
        \begin{aligned}
            \|f\|_1=\int_0^{1}|\{f\ge t\}|dt=\int_0^{1}\frac{1}{s}t^{\frac{1}{s}-1}|\{f^s\ge t\}|dt,
        \end{aligned}
    \end{equation*}
    where $\{f^s\ge t\}=(1-t)\triangle+tx_0$. Hence
    \begin{equation*}
        \|f\|_1=\int_0^{1}\frac{1}{s}t^{\frac{1}{s}-1}(1-t)^n|\triangle|dt=B(\frac{1}{s}, n+1)\frac{|\triangle|}{s}.
    \end{equation*}
    Moreover, we obtain
    \begin{equation*}
        \begin{aligned}
        g(ru)&=\int_{\rn}\min\{f(x), f(x+ru)\}dx\\
        &=\int_0^{1}|\{f\ge t\}\cap (\{f\ge t\}+ru)|dt\\
        &=\int_0^1\frac{1}{s}t^{\frac{1}{s}-1}|(1-t)\triangle\cap ((1-t)\triangle+ru)|dt\\
        &=\int_0^1\frac{1}{s}t^{\frac{1}{s}-1}(1-t)^n\Big|\triangle\cap (\triangle+\frac{r}{1-t}u)\Big|dt.
        \end{aligned}
    \end{equation*}
        By \cite[Lemma 3.4]{GZ}, $|\triangle\cap (\triangle+ ru)|^{1/n}$ is an affine function with respect to $r>0$. Therefore, there exists $\vartheta>0$ such that
    \[\Big|\triangle\cap(\triangle+\frac{r}{1-t})u\Big|=|\triangle|\Big(1-\vartheta \frac{r}{1-t}\Big)_+^n,\]
    where $\vartheta$ depends on $u$.

A straightforward calculation shows that
\[g(ru)=\frac{1}{s}B(\frac{1}{s},n+1)|\triangle|(1-\vartheta r)_+^{\frac{ns+1}{s}}.\]  
    We set $s^{\prime}=s/(ns+1)$. For $-1<\alpha<0$ and $\alpha>0$, by taking $\omega(r)=g(o)(1-s^{\prime}r)_+^{1/s^{\prime}}$ and $\varphi(r)=\frac{1}{s^{\prime}}(1-(g(ru)/g(o))^{s^{\prime}})$ in Lemma \ref{lemma_aux}, we obtain
    \[c_{n,\alpha}(s)\rho_{\Raf}(u)=\frac{1}{s^{\prime}}\zeta(\alpha)=\frac{1}{s^{\prime}}\Big(\frac{\int_0^{\infty}r^{\alpha-1}(1-\vartheta r)_+^{1/s^{\prime}}dr}{\int_0^{\infty}r^{\alpha-1}(1-s^{\prime}r)_+^{1/s^{\prime}}dr}\Big)^{1/\alpha}=\frac{1}{\vartheta},\]
    where the first equality is from the calculation in the proof of Theorem \ref{main5-section}.
    
    This implies that 
    \[c_{n,\beta}(s)\rR_{\beta} f=c_{n,\alpha}(s) \Raf,\quad\quad-1<\alpha<\beta<\infty.\]
    Letting $\alpha\rightarrow -1$ and $\beta\rightarrow\infty$ completes the proof.
\end{proof}

To establish the equality characterization, we need an auxiliary lemma
(Lemma \ref{key lemma thm5}) based on the equality case of the Borell--Brascamp--Lieb inequality \eqref{BBL}. We first fix some notation.

Let $e_1,\ldots,e_n$ be an orthonormal basis of $\rn$. We denote by
$\triangle_n$ the standard simplex, namely
\[\triangle_n=\operatorname{conv}\{o,e_1,\ldots,e_n\}=\Big\{x=(x_1,\ldots,x_n)\in\rn:\ x_i\ge0,\ \sum_{i=1}^n x_i\le1\Big\}.\]
For $u=(u_1,\ldots,u_n)\in\sn$, set
\[u_i^+=\max\{u_i,0\},\qquad u_i^-=\max\{-u_i,0\},\]
and
\begin{equation}\label{u+-}
    u^+=(u_1^+,\ldots,u_n^+),\qquad
    u^-=(u_1^-,\ldots,u_n^-).
\end{equation}
Thus $u=u^+-u^-$. We also set
\begin{equation}\label{c0}
    c_0=\max\Big\{\sum_{i=1}^n u_i^+,\sum_{i=1}^n u_i^-\Big\}.
\end{equation}

These notations will be used in the next two lemmas. The first lemma records
a simple relation between the gauge of a shifted simplex and barycentric
coordinates; the second one is a key step for Lemma \ref{key lemma thm5}.

\begin{lem}\label{simplex-gauge}
Let \(v_0=o\) and \(v_i=e_i\), \(i=1,\ldots,n\). For \(x\in\rn\), write
\[x=\sum_{i=0}^n \lambda_i(x)v_i,\qquad \sum_{i=0}^n\lambda_i(x)=1,\]
where $\lambda_0(x)=1-\sum_{i=1}^n x_i$ and $\lambda_i(x)=x_i$. 

Let \(x_0\in\triangle_n\), and set $\mathcal{I}_{x_0}=\{i:\lambda_i(x_0)>0\}$. Then, for every \(x\in\triangle_n\),
\[(1-\|x-x_0\|_{\triangle_n-x_0})_+=\min_{i\in \mathcal{I}_{x_0}}\frac{\lambda_i(x)}{\lambda_i(x_0)}.\]
\end{lem}

\begin{proof}
By the definition of the gauge function,
\[\|x-x_0\|_{\triangle_n-x_0}=\inf\{r\ge0:\ x-x_0\in r(\triangle_n-x_0)\}.\]

Suppose that $x-x_0\in r(\triangle_n-x_0)$. Then there exists
$y\in\triangle_n$ such that
\[x=x_0+r(y-x_0)=(1-r)x_0+ry.\]
Taking barycentric coordinates gives
\[\lambda_i(x)=(1-r)\lambda_i(x_0)+r\lambda_i(y), \qquad i=0,\ldots,n.\]
Since $y\in\triangle_n$, we have $\lambda_i(y)\ge0$. Hence, $\lambda_i(x)\ge (1-r)\lambda_i(x_0)$ for $i\in \mathcal{I}_{x_0}$. Therefore
\[r\ge 1-\frac{\lambda_i(x)}{\lambda_i(x_0)},\qquad i\in \mathcal{I}_{x_0}.\]
It follows that
\[r\ge1-\min_{i\in \mathcal{I}_{x_0}}\frac{\lambda_i(x)}{\lambda_i(x_0)}.\]
Taking the infimum over all admissible \(r\), we obtain
\begin{equation}\label{geq}
    \|x-x_0\|_{\triangle_n-x_0}\ge1-\min_{i\in \mathcal{I}_{x_0}}\frac{\lambda_i(x)}{\lambda_i(x_0)}.
\end{equation}

Conversely, set
\[m=\min_{i\in \mathcal{I}_{x_0}}
\frac{\lambda_i(x)}{\lambda_i(x_0)},
\qquad
r=1-m.\]
Since $x\in\triangle_n$, we have $m\ge0$, and also $m\le1$. Thus 
$r\ge0$. If $r=0$, then $m=1$, which forces $x=x_0$, and the claim is
trivial. Assume $r>0$. Define
\[\lambda_i=\frac{\lambda_i(x)-(1-r)\lambda_i(x_0)}{r},\qquad i=0,\ldots,n.\]
For $i\in \mathcal{I}_{x_0}$, the definition of $m$ gives $\lambda_i\ge0$, while for
$i\notin \mathcal{I}_{x_0}$, we have $\lambda_i(x_0)=0$ and hence
\(\lambda_i=\lambda_i(x)/r\ge0\). Moreover,
\[\sum_{i=0}^n\lambda_i=\frac{1-(1-r)}{r}=1.\]
Thus $y=\sum_{i=0}^n\lambda_i v_i\in\triangle_n$. By construction, $x=(1-r)x_0+ry$,
 thus $r$ is admissible. Therefore
\begin{equation}\label{leq}
    \|x-x_0\|_{\triangle_n-x_0}\le r=1-\min_{i\in \mathcal{I}_{x_0}}\frac{\lambda_i(x)}{\lambda_i(x_0)}.
\end{equation}
Combining \eqref{geq} and \eqref{leq} proves the lemma.
\end{proof}

\begin{lem}\label{key}
Let $f\in L^1(\rn)$ be a non-zero, non-negative, upper semicontinuous concave function with
$\supp (f)=\triangle_n$. Suppose that, for every $u\in\sn$ and every
$r\in(0,1/c_0)$,
\begin{equation}\label{cond}
    \min\Big\{f\big((1-c_0r)x+ru^-\big),f\big((1-c_0r)x+ru^+\big)\Big\}=(1-c_0r)f(x)
\end{equation}
holds for almost all $x\in\triangle_n$, where $u^+$, $u^-$, and $c_0$ are defined by \eqref{u+-} and \eqref{c0}. Then
\[f(x)=a\left(1-\|x-x_0\|_{\triangle_n-x_0}\right)_+,\]
where $a=\|f\|_\infty$ and $x_0\in\triangle_n$ is a maximum point of $f$.
\end{lem}

\begin{proof}
Let $x_0\in\triangle_n$ be a maximum point of $f$. Without loss of generality, we may assume that $a=f(x_0)=1$.

Since $f$ is finite and concave, it is continuous on
${\rm int}\triangle_n$, which denotes the interior of $\triangle_n$. Moreover, for fixed $u$ and $r$, the two
sides of \eqref{cond} are continuous functions of $x$ on
${\rm int}\triangle_n$. Hence the almost everywhere identity
\eqref{cond} extends to every $x\in{\rm int}\triangle_n$.

We also note that $f>0$ on ${\rm int}\triangle_n$. Indeed, if
$f$ vanished at an interior point, then the concavity and non-negativity of
$f$ would force $f\equiv0$ on $\triangle_n$, contradicting the assumption
that $f$ is non-zero.

Let $x\in{\rm int}\triangle_n$ be a differentiability point of
$f$. Define the supporting affine function
\[T_x(y)=f(x)+\nabla f(x)\cdot(y-x),\qquad y\in\triangle_n.\]
Since $f$ is concave,
\[f(y)\le T_x(y),\qquad y\in\triangle_n.\]
In particular, as $f\ge0$,
\[T_x(v_i)\ge0,\qquad i=0,\ldots,n.\]

\noindent{Step 1:} We first derive a restriction on $T_x$. 

Taking $u=e_i$, we have
$u^+=e_i$, $u^-=o$, and $c_0=1$. Thus, for $0<t<1$,
\[\min\{f((1-t)x+te_i),f((1-t)x)\}=(1-t)f(x).\]
Since \(f\) is differentiable at \(x\),
\[
f((1-t)x+tv)
=
(1-t)f(x)+tT_x(v)+o(t),
\qquad t\to0^+,
\]
for each fixed \(v\in\triangle_n\). Hence, after subtracting
\((1-t)f(x)\), dividing by \(t\), and letting \(t\to0^+\), we obtain
\begin{equation}\label{io}
    \min\{T_x(e_i),T_x(o)\}=0.
\end{equation}

Similarly, for $i\ne j$, take $u=\frac{e_i-e_j}{\sqrt2}$. Then 
\[
u^+=\frac{e_i}{\sqrt2},\qquad
u^-=\frac{e_j}{\sqrt2},\qquad
c_0=\frac1{\sqrt2}.\]
Putting \(t=c_0r\), \eqref{cond} gives
\[\min\{f((1-t)x+te_i),f((1-t)x+te_j)\}=(1-t)f(x).\]
Similarly, after subtracting $(1-t)f(x)$, dividing by $t$, and letting $t\to0^+$, we obtain
\begin{equation}\label{ij}
    \min\{T_x(e_i),T_x(e_j)\}=0.
\end{equation}

From \eqref{io} and \eqref{ij}, among the numbers $T_x(v_0),T_x(v_1),\ldots,T_x(v_n)$, at most one is positive. On the other hand,
\[f(x)=T_x(x)=\sum_{i=0}^n\lambda_i(x)T_x(v_i)>0,\]
and all $\lambda_i(x)$ are positive because $x\in{\rm int}\triangle_n$.
Therefore exactly one of these numbers is positive. Hence there exist
$k(x)\in\{0,\ldots,n\}$ and $A_x>0$ such that
\[T_x(y)=A_x\lambda_{k(x)}(y),\qquad y\in\triangle_n.\]

\noindent{Step 2:} We next show that $f$ is the minimum of finitely many multiples of the
barycentric coordinates. 

Let $D$ be the set of differentiability points of $f$ in ${\rm int}\triangle_n$. Since $f$ is concave,
\[f(y)\le T_x(y),\qquad x\in D,\ y\in\operatorname{int}\triangle_n.\]
Thus
\[f(y)\le \inf_{x\in D}T_x(y).\]

Conversely, since $f$ is differentiable almost everywhere, we choose \(x_j\in D\) with \(x_j\to y\). Since \(f\) is locally
Lipschitz on \(\operatorname{int}\triangle_n\), the gradients \(\nabla f(x_j)\)
are bounded for \(j\) large. Therefore
\[\begin{aligned}
|T_{x_j}(y)-f(y)|\le |f(x_j)-f(y)|+|\nabla f(x_j)|\,|y-x_j|\to0.
\end{aligned}\]
Hence
\[
f(y)=\inf_{x\in D}T_x(y),
\qquad y\in\operatorname{int}\triangle_n.
\]

For \(k=0,\ldots,n\), let
\[D_k=\{x\in D:\ T_x(y)=A_x\lambda_k(y),~\text{for all}~y\in{\rm int}\triangle_n\},\]
and define
\[
A_k=\inf_{x\in D_k}A_x,
\]
with the convention that \(A_k=+\infty\) if \(D_k=\varnothing\). Since
\(\lambda_k(y)>0\) for \(y\in\operatorname{int}\triangle_n\), we get
\[
f(y)=\min_{0\le k\le n}A_k\lambda_k(y),
\qquad y\in\operatorname{int}\triangle_n,
\]
where terms with \(A_k=+\infty\) are simply ignored. Moreover, no finite
\(A_k\) can be equal to \(0\), because \(f>0\) in
\(\operatorname{int}\triangle_n\). Thus every finite \(A_k\) is strictly
positive.

Let $\mathcal{I}=\{k:\ A_k<+\infty\}$. Then
\[f(y)=\min_{k\in \mathcal{I}}A_k\lambda_k(y),\qquad y\in\operatorname{int}\triangle_n.\]
By upper semicontinuity and concavity, this identity extends to \(\triangle_n\).

\noindent{Step 3:} We aim to show that $\lambda_i(x_0)=A_i^{-1}$.

Since \(f(x_0)=1\), we have
\[1=\min_{k\in \mathcal{I}}A_k\lambda_k(x_0).\]
Therefore
\[\lambda_k(x_0)\ge A_k^{-1},\qquad k\in \mathcal{I}.\]
Summing over \(k\in I\), we get
\[1=\sum_{k=0}^n\lambda_k(x_0)\ge\sum_{k\in \mathcal{I}}A_k^{-1}.\]

Now choose $z\in\triangle_n$ by prescribing its barycentric
coordinates as
\[\lambda_k(z)=\frac{A_k^{-1}}{\sum_{i\in\mathcal{I}}A_i^{-1}},\qquad k\in \mathcal{I},\]
and $\lambda_k(z)=0$ for $k\notin \mathcal{I}$. For every $k\in \mathcal{I}$,
\[A_k\lambda_k(z)=\frac{1}{\sum_{i\in\mathcal{I}}A_i^{-1}}.\]
Hence
\[f(z)=\frac{1}{\sum_{i\in\mathcal{I}}A_i^{-1}}\leq 1,\]
which implies that $\sum_{i\in\mathcal{I}}A_i^{-1}=1$.

Since also \(\lambda_k(x_0)\ge A_k^{-1}\) for \(k\in \mathcal{I}\) and
\(\sum_{k=0}^n\lambda_k(x_0)=1\), it follows that
\[\lambda_k(x_0)=A_k^{-1},\qquad k\in \mathcal{I},\]
and $\lambda_k(x_0)=0$ for $k\notin \mathcal{I}$. Therefore, for \(y\in\triangle_n\),
\[f(y)=\min_{k\in \mathcal{I}}A_k\lambda_k(y)=
\min_{k:\lambda_k(x_0)>0}\frac{\lambda_k(y)}{\lambda_k(x_0)}.\]
By Lemma \ref{simplex-gauge}, $f(y)=a\left(1-\|y-x_0\|_{\triangle_n-x_0}\right)_+$, which concludes the proof.
\end{proof}

We are now in a position to prove the following key result for the equality characterization in Theorem 3.

\begin{lem}\label{key lemma thm5}
    Let $s>0$ and $f\in L^1(\rn)$ be a non-zero, non-negative, upper semicontinuous $s$-concave function with compact support. For $u\in\sn$, let
    \[g(ru)=\int_{\rn}\min\{f(x), f(x+ru)\}dx.\]
    Suppose that for every $u\in\sn$, there exists $\vartheta(u)>0$ such that $g(ru)=\|f\|_1(1-\vartheta(u) r)_+^{\frac{ns+1}{s}}$ for $r>0$, then $f(x)=a(1-\|x-x_0\|_{\triangle-x_0})_+^{1/s}$, where $x_0$ is the maximizer of $f$ with $a=f(x_0)=\|f\|_{\infty}$, and $\triangle$ is a simplex in $\rn$ containing $x_0$.
\end{lem}

\begin{proof}
    We assume $\|f\|_{\infty}=1$ without loss of generality. Since $g(ru)=g(-ru)$, we set 
    \[g(ru)=\int_{\rn}\min\{f(x), f(x-ru)\}dx,\]
     and $N_f(x,r)=\min\{f(x), f(x-ru)\}$ to simplify notation. Let $r \ge 0$ and $\lambda \in [0,1]$ be fixed, and define the functions $F, G, H : \mathbb{R}^n \to \mathbb{R}$ by
    \[F(x)=N_f(x,0)=f(x),~~G(x)=N_f(x, r),~~H(x)=N_f(x,\lambda r),\]
   which then satisfy 
   \[H((1-\lambda)x+\lambda y)\ge \Big((1-\lambda)F(x)^s+\lambda G(y)^s\Big)^{1/s}=\mathcal{M}_{\lambda}^{s}(F(x), G(y))\]
   for all $x, y \in \mathbb{R}^n$ . 
    
    Since $r\mapsto g(ru)^{\frac{s}{ns+1}}$ is an affine function with respect to $r$, 
    \[\int_{\rn}H=\mathcal{M}_{\lambda}^{\frac{s}{ns+1}}\Big(\int_{\rn}F, \int_{\rn}G\Big).\]
    Hence, equality holds in the Borell-Brascamp-Lieb inequality \eqref{BBL} for $F,G,H$.

    By Theorem \ref{BBL-eq}, there exists $c_r=\Big(\frac{|\supp (G)|}{|\supp (f)|}\Big)^{1/n}$ and $y_r\in\rn$ such that
    \begin{equation}\label{g-1}
        \supp(G)=c_r\,\supp(f)+y_r,
    \end{equation}
    while the definition of $G$  gives
    \begin{equation}\label{g-2}
        \supp(G)=\supp(f)\cap (\supp(f)+ru).
    \end{equation}
    Since $u\in\sn$ is arbitrarily fixed, \eqref{g-1} and \eqref{g-2} imply that $\supp(f)\cap (\supp(f)+ru)$ is homothetic to $\supp(f)$ for all $u\in\sn$. By \cite[Lemma 3.4]{GZ},  $\supp(f)$ is a simplex denoted by $\triangle$.

   By Lemma \ref{Raf-affco}, $\Ra f$ is invariant under translations and $\Ra(f\circ\phi^{-1})=\phi\Raf$ for all $\phi\in {\rm GL}(n)$. It suffices to consider the case where $\triangle$ is the standard simplex, that is,
    \[\triangle=\triangle_n=\Big\{x=(x_1,\ldots,x_n)\in\rn: x_i\ge 0,~~\sum_{i=1}^nx_i\leq 1\Big\}.\]
   By \eqref{g-1} and \eqref{g-2}, we have
    \[c_r=1-c_0r,~~\text{and}~~y_r=ru^+,\]
    where $c_0$ is defined by \eqref{c0}. 
    
    Moreover, by Theorem \ref{BBL-eq}, there exists a $(\lambda_0,s)$-concave function $\Phi:\supp(F)\rightarrow [0,\infty)$, with $\lambda_0=\frac{\lambda c_0}{1-\lambda+\lambda c_0}$, satisfying
\begin{equation*}
        \left\{
                \begin{aligned}
                &F(x)=\Phi(x)&&,\\
                &G(c_rx+y_r)=c_r^{\frac{1}{s}}\Phi(x)&&,\\
                &H((1-\lambda+\lambda c_0)x+\lambda y_r)=\Big[\mathcal{M}_{\lambda}^{\frac{s}{ns+1 }}\Big(1, c_r^{\frac{ns+1}{s}}\Big)\Big]^{\frac{1}{ns+1}}\Phi(x)&&,
                \end{aligned}\right.
    \end{equation*}
    for almost all $x\in\triangle_n$. By the definition of $F$, we have $\Phi(x)=F(x)=f(x)$, and
    \begin{equation}\label{BBL-eq-G}
        G((1-c_0r)x+y_r)^s=\min\{f((1-c_0r)x+ru^-)^s, f((1-c_0r)x+ru^+)^s\}=(1-c_0r)f(x)^s.
    \end{equation}
    By applying Lemma \ref{key} to $f^s$, we complete the proof.
\end{proof}

As a consequence, Lemma \ref{key lemma thm5} derives the following theorem, which provides the full equality characterization in Theorem 3.

\begin{thm}
    Let $s>0$ and  $c_{n,\alpha}(s)$ be as  defined in Theorem 3. Let $f\in L^1(\rn)$ be a non-zero, non-negative, upper semicontinuous $s$-concave function with compact support. For $-1<\alpha<\beta<\infty$, if
    \[c_{n,\beta}(s)\rR_{\beta}f=c_{n,\alpha}(s)\Raf,\]
    then $f(x)=a(1-\|x-x_0\|_{\triangle-x_0})_+^{1/s}$, where $x_0$ is the maximizer of $f$ with $a=f(x_0)=\|f\|_{\infty}$, and $\triangle$ is a simplex in $\rn$ containing $x_0$.
\end{thm}

\begin{proof}
   Let $g(ru)=\int_{\rn}\min\{f(x), f(x+ru)\}dx$ for a fixed $u\in\sn$ and $r\in\mathbb{R}$. Setting  $\omega(r)=g(o)(1-s^{\prime}r)_+^{1/s^{\prime}}$ and $\varphi(r)=\frac{1}{s^{\prime}}(1-(g(ru)/g(o))^{s^{\prime}})$ in Lemma \ref{lemma_aux}, we obtain
    \[\zeta(\alpha)=s^{\prime}c_{n,\alpha}(s)\rho_{\Raf}(u),\]
    where $s^{\prime}=s/(ns+1)$.   Then, the equality  $c_{n,\beta}(s)\rR_{\beta}f=c_{n,\alpha}(s)\Raf$ implies that $\zeta(\alpha)=\zeta(\beta)$. 
    Since $\omega$ is strictly decreasing on its support, Lemma \ref{lemma_aux_iff} yields $\vartheta=\zeta(\alpha)^{-1}$ such that
    \[\varphi(r)=\frac{1}{s^{\prime}}\Big(1-\Big(\frac{g(ru)}{g(o)}\Big)^{s^{\prime}}\Big)=\vartheta r.\]
    This shows that $g(ru)^{s^\prime}$ is affine in $r$, and the conclusion follows  from Lemma~\ref{key lemma thm5}.
\end{proof}

We conclude this section by  showing the Rogers-Shephard inequality and Zhang inequality for $s$-concave functions, which are special cases of Theorem 3.

\begin{thm}[Rogers-Shephard inequality for $s$-concave functions]
    Let $s>0$ and $f\in L^1(\rn)$ be a non-zero, non-negative, $s$-concave function with compact support. Then
    \[|{\rm D}\, \supp(f)|\leq \frac{\Pi_{k=1}^n(k+\frac{1}{s})}{n\cdot n!\|f\|_1}\di\min\{f(x),f(y)\}dxdy.\]
    If $f$ is upper semicontinuous, equality holds if and only if $f(x)=a(1-\|x-x_0\|_{\triangle-x_0})_+^{1/s}$, where $x_0$ is the maximizer of $f$ with $a=f(x_0)=\|f\|_{\infty}$, and $\triangle$ is a simplex in $\rn$ containing $x_0$.
\end{thm}


\begin{thm}[Zhang's inequality for $s$-concave functions]
    Let $s>0$ and $f\in L^1(\rn)$ be a non-zero, non-negative, $s$-concave function with compact support. Then
    \[\di\min\{f(x), f(y)\}dxdy\leq \frac{n\cdot n!(n+\frac{1}{s})^n}{\Pi_{k=1}^n(k+\frac{1}{s})}\|f\|_1^{n+1}|\Pi^*f|.\]
    If $f$ is upper semicontinuous, equality holds if and only if $f(x)=a(1-\|x-x_0\|_{\triangle-x_0})_+^{1/s}$, where $x_0$ is the maximizer of $f$ with $a=f(x_0)=\|f\|_{\infty}$, and $\triangle$ is a simplex in $\rn$ containing $x_0$.
\end{thm}

\vskip 10pt

\noindent{\bf Acknowledgements:} 
This research was funded in whole or in part by the Austrian Science Fund (FWF) doi/10.55776/37030. For open access purposes, the authors have applied a CC BY public copyright license to any author accepted manuscript version arising from this submission.

\bibliographystyle{abbrv}

\end{document}